\documentclass[11pt]{article}
\usepackage[a4paper,margin=1.5cm]{geometry}
\usepackage{setspace}

\usepackage{amstext}
\usepackage{amsthm}
\usepackage{amsopn}
\usepackage{amsfonts}
\usepackage{amsmath}
\usepackage{amssymb}
\usepackage{framed}
\usepackage{enumitem}

\usepackage{comment}

\usepackage[dvipsnames]{xcolor}
\definecolor{greytext}{gray}{0.5}

\usepackage{url}
\usepackage[numbers,sort&compress]{natbib}
\renewcommand\UrlFont{\color{black!80}}

\DeclareUrlCommand\DOI{}

\usepackage[page]{appendix}
\usepackage{setspace}
\usepackage[bookmarks,bookmarksnumbered,    colorlinks,             linktocpage    ]{hyperref}

\usepackage[capitalise]{cleveref}

\newcommand\theoremlevel{subsection}
\usepackage[dvipsnames]{xcolor}
\definecolor{greytext}{gray}{0.5}

\usepackage{amstext}
\usepackage{amsthm}
\usepackage{amsopn}
\usepackage{amsfonts}
\usepackage{amsmath}
\usepackage{amssymb}
\usepackage{stackrel}
\usepackage{framed}

\usepackage{array}
\usepackage{dsfont}

\usepackage{mathrsfs}
\usepackage{mathtools}
\usepackage{enumitem}
\usepackage{tikz-cd}
\usepackage{arydshln}
\usepackage{subcaption}

\usepackage[scr = euler]{mathalpha}

\usepackage{chngcntr}
\counterwithin*{equation}{section}
\counterwithin{figure}{section}

\newtheorem{maintheorem}{Theorem}

\makeatletter
\@ifundefined{theoremlevel}{%
    \newtheorem{theorem}{Theorem}%
}{%
    \newtheorem{theorem}{Theorem}[\theoremlevel]%
}
\newtheorem{corollary}[theorem]{Corollary}
\newtheorem{lemma}[theorem]{Lemma}
\newtheorem{proposition}[theorem]{Proposition}

\theoremstyle{definition}
\newtheorem{definition}[theorem]{Definition}

\theoremstyle{remark}
\newtheorem{remark}[theorem]{Remark}

\let\oldcite\cite
\renewcommand\cite[2][]{{\rm\oldcite[#1]{#2}}}

\numberwithin{equation}{section}

\def\(#1\){$\displaystyle#1$}
\def\[#1\]{\begin{align*}#1\end{align*}}

\renewcommand\bar[1]{\overline{#1}}
\renewcommand\hat[1]{\widehat{#1}}
\renewcommand\tilde[1]{\widetilde{#1}}

\newcommand\bI{\mathbb{I}}

\newcommand\bS{\mathbb{S}}
\newcommand\Z{\mathbb{Z}}

\newcommand\bDelta{\mathbf{\Delta}}

\newcommand\bSigma{\mathbf{\Sigma}}
\newcommand\bTheta{\mathbf{\Theta}}

\newcommand\cM{\mathcal{M}}

\newcommand\sA{\mathscr{A}}
\newcommand\sB{\mathscr{B}}
\newcommand\sC{\mathscr{C}}
\newcommand\sD{\mathscr{D}}

\newcommand\sK{\mathscr{K}}

\newcommand\sP{\mathscr{P}}

\newcommand\sS{\mathscr{S}}

\newcommand\sX{\mathscr{X}}
\newcommand\sY{\mathscr{Y}}
\newcommand\sZ{\mathscr{Z}}

\newcommand\id{\operatorname{id}}

\newcommand\op{\mathrm{op}}

\newcommand\Ho{\operatorname{Ho}}
\newcommand\Hom{\operatorname{Hom}}

\newcommand\Map{\operatorname{Map}}

\newcommand\Seg{\operatorname{Seg}}
\newcommand\Sp{\operatorname{Sp}}

\newcommand\varlim{\varprojlim}
\newcommand\varcolim{\varinjlim}

\makeatletter
\newcommand{\pto}{}
\newcommand{\pgets}{}

\DeclareRobustCommand{\pto}{\mathrel{\mathpalette\p@to@gets\to}}
\DeclareRobustCommand{\pgets}{\mathrel{\mathpalette\p@to@gets\gets}}

\newcommand{\p@to@gets}[2]{%
  \ooalign{\hidewidth$\m@th#1\mapstochar\mkern5mu$\hidewidth\cr$\m@th#1\to$\cr}%
}
\makeatother

\newcommand\Cat{\mathbf{Cat}}

\newcommand\Fun{\mathbf{Fun}}
\newcommand\Grpd{\mathbf{Grpd}}

\newcommand\Pres{\mathbf{Pres}}

\newcommand\Set{\mathbf{Set}}
\newcommand\Sh{\mathbf{Sh}}

\tikzset{Rightarrow/.style={double equal sign distance,>={Implies},->},
Rrightarrow/.style={-,preaction={draw,Rightarrow}},
Rrrightarrow/.style={preaction={draw,Rightarrow,shorten >=0pt},shorten >=1pt,-,double,double
distance=0.2pt}}

\tikzcdset{
  cells={font=\everymath\expandafter{\the\everymath\displaystyle}},
}

\newcommand{\Gp}{\operatorname{Gp}}

\renewcommand{\SS}{\mathbf{SS}}
\newcommand{\CSS}{\mathbf{CSS}}
\newcommand{\Dist}{\mathbf{Dist}}

\usepackage{datetime}
\title{Sheaves of $(\infty, \infty)$-categories}
\date{\monthname~\the\year}
\author{Zach Goldthorpe}

\begin{document}
\maketitle
\begin{abstract}
    We provide a functorial presentation of the $(\infty, 1)$-category of sheaves of $(n, r)$-categories for all $-2\leq n\leq\infty$ and $0\leq r\leq n+2$ based on complete Segal space objects.
    In this definition, the equivalences of sheaves of $(\infty, \infty)$-categories are defined inductively, so we also provide a localisation at the coinductive equivalences to define the $(\infty, 1)$-category of sheaves of $\omega$-categories as well.
    
    We prove that the $(\infty, 1)$-category of sheaves of $(\infty, \infty)$-categories, and the $(\infty, 1)$-category of sheaves of $\omega$-categories both define distributors over the underlying topos of sheaves of spaces.
    Moreover, we show that these distributors define terminal and initial fixed points with respect to the construction of complete Segal space objects in a distributor.
    
    We conclude with a sheafification result: the category of sheaves of $(n, r)$-categories over a site $\sC$ can be presented as a strongly reflective localisation of $\Fun(\sC^\op, \Cat_{(n, r)})$, where the localisation functor preserves fibre products over $\Sh(\sC)$, with the analogous result also holding for sheaves of $\omega$-categories.
\end{abstract}

\tableofcontents{}

\section{Introduction}
\label{sec:intro}

A higher category is a generalisation of an ordinary category consisting of objects, $1$-morphisms between objects, $2$-morphisms between $1$-morphisms, $3$-morphisms between $2$-morphisms, \emph{ad infinitum}.
For $-2\leq n\leq\infty$ and $0\leq r\leq n+2$, an \emph{$(n, r)$-category} is a higher category where
\begin{itemize}
    \item parallel $k$-morphisms are equivalent for $k > n$, and
    \item $k$-morphisms are invertible for $k > r$.
\end{itemize}
For $r < \infty$, the notion of $(n, r)$-category is unambiguous up to an action of $(\Z/2)^r$ by \cite{barwick-schommer-pries}.
However, there are several inequivalent notions of $(\infty, \infty)$-category, due to nuances in the concept of invertibility in a purely infinite-dimensional setting.

We have a tower of inclusions
$$
\Cat_{(\infty, 0)} \subseteq \Cat_{(\infty, 1)} \subseteq \Cat_{(\infty, 2)} \subseteq \dots
$$
where each inclusion $\Cat_{(\infty, r)}\subseteq\Cat_{(\infty, r+1)}$ admits a left adjoint $\pi_{\leq r}$ and a right adjoint $\kappa_{\leq r}$.
For an $(\infty, r+1)$-category $\sC$, we obtain the $(\infty, r)$-category $\pi_{\leq r}\sC$ by formally inverting the $(r+1)$-morphisms of $\sC$.
On the other hand, $\kappa_{\leq r}\sC$ is just the maximal sub-$(\infty, r)$-category of $\sC$, obtained by discarding the non-invertible $(r+1)$-morphisms.
From these adjoints, we obtain two natural constructions of the $(\infty, 1)$-category of $(\infty, \infty)$-categories:
\[
    \Cat_{(\infty, \infty)} &:= \varlim\left(\dots \to \Cat_{(\infty, 2)} \xrightarrow{\kappa_{\leq1}} \Cat_{(\infty, 1)} \xrightarrow{\kappa_{\leq0}} \Cat_{(\infty, 0)}\right) \\
    \Cat_\omega &:= \varlim\left(\dots \to \Cat_{(\infty, 2)} \xrightarrow{\pi_{\leq1}} \Cat_{(\infty, 1)} \xrightarrow{\pi_{\leq0}} \Cat_{(\infty, 0)}\right)
\]
Roughly speaking, $\Cat_{(\infty, \infty)}$ consists of higher categories wherein equivalences are defined inductively from the identity $k$-morphisms, whereas $\Cat_\omega$ consists of higher categories wherein equivalences are defined coinductively; see \cite[Remark 3.0.5]{goldthorpe:fixed}.

There are several approaches to defining higher categories, but they are largely guided by one of the two principles:
\begin{itemize}
    \item An $(n+1, r+1)$-category is a category that is suitably \emph{enriched} in $(n, r)$-categories; that is, an $(n+1, r+1)$-category is a suitable category $\sC$ such that $\Hom_\sC(x, y)$ is an $(n, r)$-category for every $x, y\in\sC$.
    \item An $(n+1, r+1)$-category is a suitable category object \emph{internal} to $(n, r)$-categories; that is, an $(n+1, r+1)$-category consists of an $(n, r)$-category $\sC_0$ of objects, and an $(n, r)$-category $\sC_1$ of morphisms, such that $(\sC_0, \sC_1)$ satisfies a suitable analogue of the axioms of category theory.
\end{itemize}
Although enrichment provides a more concise inductive description of higher categories---$(n+1, r+1)$-categories should be precisely the categories enriched in $(n, r)$-categories---a suitably general theory of enrichment proved too difficult to develop without already establishing a sufficiently general theory of higher categories.

On the other hand, a suitable theory of internalisation was more readily available using the tools from abstract homotopy theory.
The caveat is that a general category object $\sC_\bullet$ in $\Cat_{(n, r)}$ has more structure than that of an $(n+1, r+1)$-category.
Between objects in $\sC_0$, there are two distinct notions of a morphism between them: there are the objects of $\sC_1$, and there are also the morphisms of $\sC_0$.

One way to avoid this caveat is to assert that $\sC_0$ is a set; that is, $\sC_0$ is a $(0, 0)$-category.
We can then think of $\sC_0$ as the underlying set of objects of the $(n+1, r+1)$-category.
This leads to the notion of higher \emph{Segal categories}, which is developed in detail, for instance, in \cite{simpson}.
However, the ``underlying set'' of objects in an ordinary category is not invariant under equivalence, and this complicates the notion of equivalence between higher Segal categories.
A functor $\sC_\bullet\to\sD_\bullet$ of $(n+1, r+1)$-categories should be an equivalence precisely if it is fully faithful on the $(n, r)$-categories of morphisms, and essentially surjective on the objects.
Since essential surjectivity cannot be tested on the underlying functor $\sC_0\to\sD_0$, it must be defined by extracting the higher equivalences from the categorical structure of $\sC_\bullet$ and $\sD_\bullet$.

Alternatively, we can instead assert that $\sC_0$ is an $\infty$-groupoid; that is, all $k$-morphisms in $\sC_0$ are invertible.
Then, $\sC_0$ serves as the underlying $n$-groupoid of $\sC_\bullet$, which is an invariant of $\sC_\bullet$ under equivalence.
By the Homotopy Hypothesis, we can therefore view $\sC_0$ as a space.
This leads to the notion of higher \emph{Segal spaces}.
However, this means that the objects of $\sC_\bullet$ have two distinct notions of equivalence:
\begin{itemize}
    \item Path-connectedness in the space $\sC_0$, and
    \item Isomorphism from the categorical structure of $\sC_\bullet$.
\end{itemize}
Therefore, we restrict to \emph{complete Segal spaces}, which are those Segal spaces for which these two notions of equivalence coincide.
Although the model for $(n+1, r+1)$-categories becomes more technical, requiring the completeness condition, the resulting equivalences are easier to define: a functor $\sC_\bullet\to\sD_\bullet$ of complete Segal spaces modelling $(n+1, r+1)$-categories is an equivalence precisely if $\sC_0\to\sD_0$ is a homotopy equivalence (giving essential surjectivity), and $\sC_1\to\sD_1$ is an equivalence of $(n, r)$-categories (giving fully faithfulness).

Complete Segal spaces were first defined by Rezk in \cite{rezk:css}.
Barwick then defined $r$-fold iterated complete Segal spaces as a model for $(\infty, r)$-categories in \cite{barwick} (see \cite[\S14]{barwick-schommer-pries}).
Lurie expanded this construction and axiomatised its iterative nature in \cite[\S1]{lurie:infty2}, providing a means to construct complete Segal space objects in any \emph{distributor}.

Intuitively, a distributor $\sY$ is a locally presentable $(\infty, 1)$-category with a choice of full subcategory $\sX\subseteq\sY$ consisting of the ``spaces.''
Then, the $(\infty, 1)$-category $\CSS_\sX(\sY)$ of complete Segal space object in $\sY$ consists of categories $\sC_\bullet$ internal to $\sY$ such that $\sC_0\in\sX$ (i.e., $\sC_0$ is ``space-like''), and $\sC_\bullet$ is complete.
The axioms of a distributor imply that the subcategory $\sX$ of ``spaces'' is always an $(\infty, 1)$-topos; that is, an $(\infty, 1)$-category of sheaves of spaces.
In particular, $r$-fold complete Segal spaces over $\sX$ define a model for sheaves of $(\infty, r)$-categories.

In this paper, we use the language of complete Segal spaces to define a presentation of the $(\infty, 1)$-category of sheaves of $(n, r)$-categories over an $(\infty, 1)$-topos $\sX$ as a localisation of $r$-fold simplicial objects in $\sX$.
We then show that the $(\infty, 1)$-categories of sheaves of $(\infty, \infty)$-categories forms a distributor over $\sX$.
In fact, the category of sheaves of $(\infty, \infty)$-categories over $\sX$ is completely characterised by the construction of complete Segal spaces over $\sX$.
More precisely, we construct an $(\infty, 1)$-category $\Dist_\sX^R$ of distributors over $\sX$ and show that sheaves of $(\infty, \infty)$-categories satisfies a universal property with respect to the construction of complete Segal space objects:
\begin{maintheorem}[\cref{thm:CSS-terminal}]\label{thm:A}
    The $(\infty, 1)$-category of sheaves of $(\infty, \infty)$-categories is the terminal object in the $(\infty, 1)$-category $\mathbf{Fix}_{\CSS}^R$ of distributors $\sY$ over $\sX$ with an equivalence $\sY\xrightarrow\sim\CSS_\sX(\sY)$.
\end{maintheorem}

\subsection{Organisation of paper}
\label{ssec:org}

\Cref{sec:sheaves} focuses on an explicit presentation of sheaves of higher categories.
In \cref{ssec:poly}, we define polysimplicial sheaves to serve as the underlying structure of sheaves of higher categories.
We then provide an explicit definition of sheaves of $(n, r)$-categories for $-2 \leq n\leq\infty$ and $0\leq r\leq n+2$ in terms of polysimplicial sheaves in \cref{ssec:hco}.
In \cref{ssec:sh-infty}, we prove that sheaves of $(\infty, \infty)$-categories can be constructed as a suitable limit of sheaves of $(\infty, r)$-categories, analogous to the construction of $\Cat_{(\infty, \infty)}$.

In \cref{sec:distributors}, we compare the constructions of \cref{sec:sheaves} to the construction of complete Segal spaces in distributors.
We recall the preliminary theory of distributors and complete Segal spaces in \cref{ssec:prelim}.
Then, in \cref{ssec:dist}, we construct the $(\infty, 1)$-category $\Dist_\sX^R$, and prove that it has all small limits and colimits.
We show that the construction of complete Segal spaces is functorial over $\Dist_\sX^R$ in \cref{ssec:css}, and moreover that it preserves weakly contractible limits.
In \cref{ssec:univ}, we show that the constructions in \cref{sec:sheaves} can be expressed in terms of complete Segal space objects in distributors, and prove \cref{thm:A}.
Then, we briefly explore the functoriality of the construction of sheaves of higher categories in \cref{ssec:sheafification}, and moreover prove that presheaves of higher categories admit a relatively left exact sheafification functor.

\subsection{Notation and terminology}
\label{ssec:notation}

In this paper, unless otherwise specified, unprefixed categorical notions refer to their $(\infty, 1)$-categorical analogues: ``categories'' are $(\infty, 1)$-categories, ``topoi'' are $(\infty, 1)$-topoi, and ``(co)limits'' are $(\infty, 1)$-(co)limits.

Let $\bDelta$ denote the simplex category, and let $\sS$ denote the category of homotopy types.
Given a small category $\sC$, let $\sP(\sC) := \Fun(\sC^\op, \sS)$ denote the category of presheaves on $\sC$.

The categories $\Cat_{(n, r)}$ for $-2\leq n\leq\infty$ and $0\leq r\leq n$ and $\Cat_\omega$ refer to the categories of (small) higher categories defined in \cite[Definitions 3.0.1 and 3.0.4]{goldthorpe:fixed}.
On the other hand, let $\hat{\Cat_\infty}$ denote the huge category of large categories.

For a regular cardinal $\lambda$, say that a functor is \emph{$\lambda$-accessible} if it preserves $\lambda$-filtered colimits.
In particular, say that a functor is \emph{accessible} if it is $\omega$-accessible.
A category $\sC$ is said to be a \emph{$\lambda$-accessible localisation} of $\sD$ if $\sC$ embeds fully faithfully into $\sD$ via a $\lambda$-accessible right adjoint, and the localisation is moreover said to be \emph{left exact} if the corresponding left adjoint is left exact.

For a regular cardinal $\lambda$, say that a category is \emph{locally $\lambda$-presentable} if it is a $\lambda$-accessible localisation of a presheaf category (which are called ``$\lambda$-compactly generated'' categories in \cite{lurie}), and say that a category is \emph{locally presentable} if it is locally $\lambda$-presentable for some $\lambda\gg0$ (which are called ``presentable'' categories in \cite{lurie}).
Denote by $\Pres_\infty^L$ or $\Pres_\infty^R$ the huge categories of locally presentable categories, and left or right adjoints, respectively.

Say that a functor $F : \sC \to \sD$ \emph{creates limits} if any $p : K^\triangleleft\to\sC$ is a limit cone if and only if $F\circ p$ is a limit cone.
The concept of creating colimits is entirely dual.

\section{Sheaves of higher categories}
\label{sec:sheaves}

Throughout this section, fix a topos $\sX$, presented as a left exact accessible localisation of $\sP(\sC)$ for some small category $\sC$, where the localisation functor is denoted $(-)^\#:\sP(\sC)\to\sX$.
The goal of this section is to explicitly construct the category $\Sh_{(n, r)}(\sC)$ of sheaves of $(n, r)$-categories over $\sC$, where $-2\leq n\leq\infty$ and $0\leq r\leq n+2$, as well as the category $\Sh_\omega(\sC)$ of sheaves of $\omega$-categories.

\subsection{Polysimplicial sheaves}
\label{ssec:poly}

\begin{definition}\label{def:poly}
    Define the \emph{polysimplex category} $\bSigma_\infty$ as the following subcategory of $\bDelta^{\times\infty}$.
    \begin{itemize}
        \item The objects of $\bSigma_\infty$ are infinite sequences $[\vec k] = (k_0, k_1, k_2, \dots)$ of simplices such that there exists some $r\geq0$ such that $k_n = 0$ if and only if $n\geq r$.
            Call $r$ the \emph{dimension} of $[\vec k]$, denoted $\dim[\vec k] := r$.
        \item The morphisms $\vec\phi : [\vec k]\to[\vec\ell]$ are sequences of morphisms $\phi_n : [k_n]\to[\ell_n]$ in $\bDelta$ such that there exists some $r\geq0$ such that $\phi_n = 0$ if and only if $n\geq r$.
    \end{itemize}
    The composition of two morphisms $[\vec k]\xrightarrow{\vec\phi}[\vec\ell]\xrightarrow{\vec\psi}[\vec m]$ is given as follows.
    Let $r\geq0$ be minimal such that $\psi_r\phi_r$ is constant.
    Then,
    $$
    (\vec\psi\circ\vec\phi)_r = \begin{cases} \psi_n\phi_n, & \text{if $n\leq r$} \\ \psi_r\phi_r, & \text{otherwise} \end{cases}
    $$
    
    Denote by $\bSigma_r$ the full subcategory of $\bSigma_\infty$ spanned by polysimplices of dimension $\leq r$.
\end{definition}
\begin{remark}\label{rem:poly-quot}
    The category $\bSigma_r$ is equivalent to the quotient category $\bTheta^r$ of $\bDelta^{\times r}$ defined in \cite[\S2]{simpson:svk}.
    We choose different notation so as to not confuse the polysimplex category with Joyal's disk category.
\end{remark}

\begin{definition}\label{def:trunc}
    For a polysimplex $[\vec k]$ and $r\geq0$, define the polysimplices
    \[
        [\vec k]_{<r} &:= (k_0, k_1, k_2, \dots, k_{r-1}, 0, \dots) \\
        [\vec k]_{\geq r} &:= (k_r, k_{r+1}, k_{r+2}, \dots) \\
    \]
    These constructions extend to define endofunctors on $\bSigma_\infty$.
    Note that $[-]_{<r}$ restricts to a right adjoint to the inclusion $\bSigma_r\subseteq\bSigma_\infty$.

    Given two polysimplices $[\vec k]$ and $[\vec\ell]$, let $[\vec k,\vec\ell]$ denote their concatenation; that is, the unique polysimplex such that $[\vec k,\vec\ell]_{<r} = [\vec k]$ and $[\vec k,\vec\ell]_{\geq r}=[\vec\ell]$, where $r = \dim[\vec k]$.
\end{definition}

\begin{definition}\label{def:poly-sheaf}
    A \emph{polysimplicial sheaf} over $(\sC, (-)^\#)$ is a polysimplicial object in $\sX$; that is, a functor $\bSigma_\infty^\op\to\sX$.

    For a polysimplicial sheaf $X:\bSigma_\infty^\op\to\sX$ and a polysimplex $[\vec k]$, let $\Delta[\vec k, X] : \bSigma_\infty^\op\to\sX$ denote the polysimplicial sheaf given by the left Kan extension of $X$ along the functor $[\vec k,-]:\bSigma_\infty\to\bSigma_\infty$.
    By definition, it follows for all polysimplicial sheaves $X'$ that we have a natural equivalence
    $$
    \Map(\Delta[\vec k, X], X') \simeq \Map(X, X'_{\vec k,\bullet})
    $$
    
    We are also interested in a couple of special cases:
    \begin{itemize}
        \item For a polysimplex $[\vec k]$ and object $U\in\sC$, define the \emph{$U$-local $\vec k$-polysimplex} to be $\Delta_U[\vec k] := \Delta[\vec k,h_U^\#]$, which corepresents the functor $\Fun(\bSigma_\infty^\op,\sX)\to\sS$ mapping $X_\bullet\mapsto X_{\vec k}(U)$.
        \item Define the \emph{suspenion} of a polysimplicial sheaf $X$ to be $\Sigma X := \Delta[1, X]$, which is left adjoint to the \emph{desuspension} $\Omega X' := X'_{1,\bullet}$.
    \end{itemize}
\end{definition}
\begin{remark}\label{rem:zero-simplex}
    The functor $[0,-] : \bSigma_\infty \to \bSigma_\infty$ is constant on the terminal object $[0]\in\bSigma_\infty$.
    In particular, the polysimplicial sheaf $\Delta[0, X]$ is the terminal polysimplicial sheaf, so we may denote this object by $\Delta[0]$.
    
    On the other hand, the $U$-local $0$-polysimplex $\Delta_U[0]$ is \emph{not} the terminal object; rather, it is the corepresenting object for the functor mapping $X_\bullet\mapsto X_0(U)$.
    This is precisely the representable sheaf $\Delta_U[0]\simeq h_U^\#$.
\end{remark}
\begin{remark}\label{rem:poly-sheaf-gen}
    Since $\Fun(\bSigma_\infty^\op,\sX)$ is a localisation of $\sP(\bSigma_\infty\times\sC)$, every polysimplicial sheaf is a canonical colimit of the local polysimplices $\Delta_U[\vec k]$.
\end{remark}
\begin{remark}\label{rem:r-poly-sheaf}
    Left Kan extension along $\bSigma_r\subseteq\bSigma_\infty$ defines a fully faithful functor $\Fun(\bSigma_r^\op,\sX) \subseteq \Fun(\bSigma_\infty^\op,\sX)$.
    We may refer to objects in the essential image of this inclusion as $r$-polysimplicial sheaves.

    Explicitly, we view $X : \bSigma_r^\op\to\sX$ as an $r$-polysimplicial sheaf by taking $X_{\vec k} := X_{\vec k_{<r}}$ for every $[\vec k]\in\bSigma_\infty$.
\end{remark}

We now establish some useful preliminary results before studying higher category objects in $\sX$.

\begin{definition}\label{def:spine}
    For a polysimplex $[\vec k]$ and an object $U\in\sC$, define the \emph{$U$-local $\vec k$-spine} to be the pushout
    $$
    \Sp_U[\vec k] := \underbrace{\Delta_U[1, \vec k_{\geq1}]\sqcup_{\Delta[0]}\dots\sqcup_{\Delta[0]}\Delta_U[1, \vec k_{\geq1}]}_{k_0}
    $$
    in $\Fun(\bSigma_\infty^\op,\sX)$, induced by the inert maps $[0],[1]\hookrightarrow[k_0]$ in $\bDelta$.
    In particular, the $U$-local $\vec k$-spine admits a canonical inclusion $\Sp_U[\vec k]\subseteq\Delta_U[\vec k]$.
\end{definition}

Recall from \cite[Definition 5.5.4.5]{lurie} that a class $W$ of morphisms in a category $\sD$ is \emph{strongly saturated} if it is closed under pushouts along arbitrary morphisms of $\sD$, it is closed under colimits in $\Fun([1], \sD)$, and it satisfies 2-out-of-3.

\begin{lemma}\label{lem:spine}
    For $0\leq r\leq\infty$, let $X : \bSigma_r^\op\to\sX$.
    If $[\ell] = [\ell^1]\sqcup_{[0]}[\ell^2]$ in $\bDelta$, then the inclusion
    $$
    \Delta[\ell^1,X]\sqcup_{\Delta[0]}\Delta[\ell^2,X] \hookrightarrow \Delta[\ell, X]
    $$
    lies in the strongly saturated class of morphisms generated by the local spine inclusions in $\Fun(\bSigma_{r+1}^\op,\sX)$.
\end{lemma}
\begin{proof}
    By definition, the strongly saturated class of morphisms generated by the local spine inclusions is closed under colimits in $\Fun([1], \Fun(\bSigma_{r+1}^\op,\sX))$.
    Therefore, \cref{rem:poly-sheaf-gen} allows us to reduce to the case where $X \simeq \Delta_U[\vec k]$.
    In this situation, consider the composite
    $$
    \begin{tikzcd}
        \Sp_U[\ell, \vec k] \ar[r, "\sim"]
        & \Sp_U[\ell^1, \vec k]\sqcup_{\Delta[0]}\Sp_U[\ell^2, \vec k] \ar[r, hook, "i"]
        & \Delta_U[\ell^1, \vec k]\sqcup_{\Delta[0]}\Delta[\ell^2, \vec k] \ar[r, hook, "j"]
        & \Delta_U[\ell, \vec k]
    \end{tikzcd}
    $$
    The map $j$ is precisely the inclusion $\Delta[\ell^1, X]\sqcup_{\Delta[0]}\Delta[\ell^2, X] \hookrightarrow \Delta[\ell, X]$.
    Note that the composite is a local spine inclusion, and $i$ is a pushout of local spine inclusions.
    Therefore, $j$ is in the strongly saturated class by 2-out-of-3.
\end{proof}

\begin{proposition}\label{prop:spine-redux}
    For $0\leq r\leq\infty$, let $W$ be any class of morphisms in $\Fun(\bSigma_r^\op,\sX)$.
    Then, the following classes generate the same strongly saturated class of morphisms in $\Fun(\bSigma_{r+1}^\op,\sX)$:
    \begin{enumerate}[label={(\arabic*)}]
        \item\label{it:W1}
            $\left\{\Delta[\ell^1,X]\sqcup_{\Delta[0]}\Delta[\ell^2, X]\hookrightarrow\Delta[\ell, X] :\middle|: [\ell] = [\ell^1]\sqcup_{[0]}[\ell^2]; X:\bSigma_r^\op\to\sX\right\}\cup\big\{\Delta[\ell, f] :\big|: f\in W; \ell > 0\big\}$
        \item\label{it:W2}
            $\left\{\Sp_U[\vec k] \hookrightarrow \Delta_U[\vec k] :\middle|: [\vec k]\in\bSigma_{r+1}; U\in\sC\right\}\cup\big\{\Sigma f :\big|: f\in W\big\}$
    \end{enumerate}
\end{proposition}
\begin{proof}
    Let $W_1$ and $W_2$ denote the strongly saturated classes of morphisms generated by \ref{it:W1} and \ref{it:W2}, respectively, then clearly $W_2\subseteq W_1$.

    Conversely, any $\Delta[\ell^1,X]\sqcup_{\Delta[0]}\Delta[\ell^2, X]\hookrightarrow\Delta[\ell, X]$ lies in $W_2$ by \cref{lem:spine}.
    For a polysimplicial sheaf $X_\bullet$, let
    \[
        \Sp[\ell, X] := \underbrace{\Sigma X\sqcup_{\Delta[0]}\dots\sqcup_{\Delta[0]}\Sigma X}_\ell \hookrightarrow \Delta[\ell, X]
    \]
    then this inclusion also lies in $W_2$.

    For $f:X_\bullet\to X_\bullet'$ in $W$, the induced map $\Sp[\ell, X]\to\Sp[\ell, X']$ is a pushout of elements of $W_2$ and is thus an element of $W_2$ as well.
    From the commutative square
    $$
    \begin{tikzcd}
        \Sp[\ell, X] \ar[r, hook, "\in W_2"]\ar[d, "\in W_2"']
        & \Delta[\ell, X] \ar[d, "{\Delta[\ell, f]}"] \\
        \Sp[\ell, X'] \ar[r, hook, "\in W_2"']
        & \Delta[\ell, X']
    \end{tikzcd}
    $$
    it follows from 2-out-of-3 that $\Delta[\ell, f]\in W_2$, as desired.
\end{proof}

\subsection{Higher category objects in sheaves}
\label{ssec:hco}

\begin{definition}\label{def:loc-realisation}
    For $U\in\sC$, left Kan extension of $\Delta_U[-] : \bDelta \to \Fun(\bSigma_r^\op,\sX)$ for any $0\leq r\leq\infty$ along the Yoneda embedding $\bDelta\hookrightarrow\sP(\bDelta)$ induces a $U$-local realisation functor $|{-}|_U^\# : \sP(\bDelta) \to \Fun(\bSigma_r^\op,\sX)$.
    By definition, this realisation is left adjoint to the functor $\Fun(\bSigma_r^\op,\sX) \to \sP(\bDelta)$ sending $X_\bullet\mapsto X_\bullet(U)$.
\end{definition}

\begin{definition}\label{def:infty-sheaf}
    Let $\bI : \bDelta^\op\to\Set$ denote the nerve of the walking isomorphism.
    Explicitly, we have $\bI_n = \{\bot, \top\}^{\times(n+1)}$ for all $n\geq0$, where the face maps are given by projections, and the degeneracy maps are given by diagonal inclusions.
    On the other hand, let $\bS^m := \partial\Delta[m+1]$ denote the simplicial $m$-sphere for $m\geq-1$.
    
    For $-2\leq n\leq\infty$ and $0\leq r\leq n+2$, let $W_{(n, r)}$ denote the strongly saturated class of morphisms in $\Fun(\bSigma_\infty^\op, \sX)$ generated by:
    \begin{enumerate}[label={(W\arabic*)}]
        \item\label{W:segal}
            $\Sigma^m\Sp_U[\vec k]\hookrightarrow\Sigma^m\Delta_U[\vec k]$ for $U\in\sC$, $m\geq0$, and $[\vec k]\in\bSigma_\infty$,
        \item\label{W:complete}
            $\Sigma^m|\bI|_U^\# \to \Sigma^m\Delta_U[0]$ for $U\in\sC$, and $m\geq0$,
        \item\label{W:truncate}
            $\Sigma^m(\bS^{k-m}\times\Delta_U[0])\to\Sigma^m\Delta_U[0]$ for $U\in\sC$, and finite $m \geq r$ and $k > n$, and
        \item\label{W:trivial}
            $\Sigma^m\Delta_U[1] \to \Sigma^m\Delta_U[0]$ for $U\in\sC$, and finite $m \geq r$.
    \end{enumerate}
    Call a polysimplicial sheaf a \emph{sheaf of $(n, r)$-categories} on $\sC$ if it is $W_{(n, r)}$-local.
    Denote by $\Sh_{(n, r)}(\sC)$ the full subcategory of $\Fun(\bSigma_\infty^\op,\sX)$ spanned by the sheaves of $(n, r)$-categories.
\end{definition}
\begin{remark}\label{rem:model-structure}
    Following the proof of \cite[Proposition 1.5.4]{lurie:infty2}, suppose $\cM$ is a left proper combinatorial simplicial model category such that its underlying $\infty$-category is a topos $\sX$.
    For $-2\leq n\leq\infty$ and $0\leq r\leq n+2$, we can endow $\Fun(\bSigma_\infty^\op,\cM)$ with a simplicial model structure such that
    \begin{itemize}
        \item[(C)] A morphism $f:X_\bullet\to X_\bullet'$ is a cofibration precisely if the map $X_{\vec k}\to X'_{\vec k}$ is a cofibration in $\cM$ for every $[\vec k]\in\bSigma_\infty$,
        \item[(W)] pointwise weak equivalences are among the weak equivalences in $\Fun(\bSigma_\infty^\op,\cM)$, and
        \item[(F)] An object $X$ is fibrant if and only if it is injectively fibrant, and the induced map $N(\bSigma_\infty)^\op\to N(X)$ in the underlying $\infty$-category $\Fun(\bSigma_\infty^\op,\sX)$ is a sheaf of $(n, r)$-categories.
    \end{itemize}
    In fact, by \cref{lem:truncate} below, we can restrict to a Quillen equivalent model structure on the subcategory $\Fun(\bSigma_r^\op, \cM)$.
\end{remark}
\begin{remark}\label{rem:infty-sheaf}
    We can generalise \cref{def:loc-realisation} to any $X:\bSigma_r^\op\to\sX$ by taking the left Kan extension of $\Delta[-, X] : \bDelta \to \Fun(\bSigma_{r+1}^\op,\sX)$ along the Yoneda embedding $\bDelta\hookrightarrow\sP(\bDelta)$.
    This defines an action $(-)\odot X : \sP(\bDelta) \to \Fun(\bSigma_{r+1}^\op,\sX)$.
    
    By \cref{rem:poly-sheaf-gen}, we have for any simplicial space $S$ that locality with respect to $\Sigma^m|S|_U^\#\to\Sigma^m\Delta_U[0]$ for all $U\in\sC$ is equivalent to locality with respect to $\Sigma^m(S\odot X)\to\Sigma^mX$ for all polysimplicial sheaves $X$.    
    In particular, combined with \cref{prop:spine-redux}, we see that $W_{(n, r)}$ is independent of the choice of site $\sC$.
    Therefore, we may also denote the category of sheaves of $(n, r)$-categories on (a site of) $\sX$ by $\Sh_{(n, r)}(\sX)$ to avoid specifying a site of definition.
\end{remark}

Below, we provide a recursive characterisation of sheaves of $(n, r)$-categories.

\begin{definition}\label{def:uncurry}
    For $0\leq r\leq\infty$, we have a functor $\bDelta\times\bSigma_r\to\bSigma_{r+1}$ given by
    $$
    ([n], [\vec k]) \mapsto \begin{cases} [0], & n = 0 \\ [n, \vec k], & n > 0 \end{cases}
    $$
    induces a fully faithful inclusion $\Fun(\bSigma_{r+1}^\op, \sX) \hookrightarrow \Fun(\bDelta^\op, \Fun(\bSigma_r^\op, \sX))$.
\end{definition}

\begin{proposition}\label{prop:infty-sheaf}
    For $-2\leq n\leq\infty$ and $0\leq r\leq n+2$, the essential image of the fully faithful inclusion
    \[
        \Sh_{(n+1, r+1)}(\sC) \subseteq \Fun(\bSigma_\infty^\op, \sX) \hookrightarrow \Fun(\bDelta^\op, \Fun(\bSigma_\infty^\op, \sX))
    \]
    consists of those $F : \bDelta^\op \to \Fun(\bSigma_\infty^\op,\sX)$ such that
    \begin{enumerate}[label={(S\arabic*)}]
        \item\label{S:obj}
            $F_0 : \bSigma_\infty^\op\to\sX$ is essentially constant,
        \item\label{S:mor}
            $F_1$ is a sheaf of $(n, r)$-categories over $\sC$,
        \item\label{S:segal}
            $F$ satisfies the \emph{Segal condition}: for every $k\geq2$, the canonical map
            $$
            F_k \to \underbrace{F_1\times_{F_0}\dots\times_{F_0}F_1}_k
            $$
            is an equivalence in $\Fun(\bSigma_\infty^\op,\sX)$, and
        \item\label{S:complete}
            $F$ is \emph{Rezk-complete}: it is local with respect to $|\bI|_U^\#\to\Delta_U[0]$ for every $U\in\sC$.
    \end{enumerate}
\end{proposition}
\begin{remark}\label{rem:infty-sheaf-recurse}
    If $F$ satisfies \ref{S:obj}, \ref{S:mor}, and \ref{S:segal} and $n=\infty$, then $F_k$ is a sheaf of $(\infty, r)$-categories over $\sC$ for every $k\geq0$.
    In particular, we can characterise $\Sh_{(\infty, r+1)}(\sC)$ as the full subcategory of $\Fun(\bDelta^\op,\Sh_{(\infty, r)}(\sC))$ spanned by those $F: \bDelta^\op \to \Sh_{(\infty, r)}(\sC)$ such that
    \begin{itemize}
        \item $F_0$ is essentially constant,
        \item $F$ satisfies the Segal condition, and
        \item $F$ is Rezk-complete.
    \end{itemize}
\end{remark}
\begin{proof}[Proof of \cref{prop:infty-sheaf}.]
    A functor $F : \bDelta^\op\times\bSigma_\infty^\op\to\sX$ factors through the left adjoint $\bDelta\times\bSigma_\infty\to\bSigma_\infty$ of \cref{def:uncurry} if and only if $F$ satisfies \ref{S:obj}.
    Assuming $F$ satisfies \ref{S:obj}, we therefore tacitly identify $F$ with its underlying functor $F : \bSigma_\infty^\op\to\sX$.
    
    By definition, \ref{S:mor} is equivalent to asserting that $F$ is local with respect to $\Sigma(W_{(n, r)})$; that is, $F$ is local with respect to:
    \begin{enumerate}[label={(W\arabic*')}]
        \item
            $\Sigma^{m+1}\Sp_U[\vec k]\hookrightarrow\Sigma^{m+1}\Delta_U[\vec k]$ for $U\in\sC$, $m\geq0$, and $[\vec k]\in\bSigma_\infty$,
        \item
            $\Sigma^{m+1}|\bI|_U^\# \to \Sigma^{m+1}\Delta_U[0]$ for $U\in\sC$, and $m\geq0$,
        \item
            $\Sigma^{m+1}(\bS^{k-m}\times\Delta_U[0])\to\Sigma^{m+1}\Delta_U[0]$ for $U\in\sC$, and finite $m \geq r$ and $k > n$, and
        \item
            $\Sigma^{m+1}\Delta_U[1] \to \Sigma^{m+1}\Delta_U[0]$ for $U\in\sC$, and finite $m \geq r$.
    \end{enumerate}
    all of which are a subset of the generators for $W_{(n+1, r+1)}$.
    In particular, $F$ is the image of a sheaf of $(n+1, r+1)$-categories if and only if $F$ satisfies \ref{S:obj} and \ref{S:mor}, and is moreover local with respect to
    \begin{itemize}
        \item $\Sp_U[\vec k]\hookrightarrow\Delta_U[\vec k]$ for $[\vec k]\in\bSigma_r$ and $U\in\sC$, and
        \item $|\bI|_U^\#\to\Delta_U[0]$ for $U\in\sC$.
    \end{itemize}
    Locality with respect to $|\bI|_U^\#\to\Delta_U[0]$ is precisely condition \ref{S:complete}.
    It remains to show that \ref{S:segal} is equivalent to locality with respect to the remaining spine inclusions.
    
    The Segal condition \ref{S:segal} is equivalent to locality with respect to $\Sp[k, X]\hookrightarrow\Delta[k, X]$ for every $k\geq2$ and all polysimplicial sheaves $X$.
    By \cref{rem:poly-sheaf-gen}, it suffices to assert locality with respect to $\Sp_U[\vec k] \hookrightarrow\Delta_U[\vec k]$ for $U\in\sC$ and all $[\vec k]\in\bSigma_\infty$.
    Assuming \ref{S:mor}, locality with respect to $\Sigma^{m+1}\Delta_U[1]\to\Sigma^{m+1}\Delta_U[0]$ for all finite $m\geq r$ implies that the Segal condition is equivalent to locality with respect to $\Sp_U[\vec k]\hookrightarrow\Delta_U[\vec k]$ for $U\in\sC$ and $[\vec k]\in\bSigma_r$, as desired.
\end{proof}

\begin{lemma}\label{lem:truncate}
    For $0\leq r < \infty$, the category $\Sh_{(\infty, r)}(\sC)$ is precisely the full subcategory of $\Sh_{(\infty, \infty)}(\sC)$ spanned by those $X:\bSigma_\infty^\op\to\sX$ that factor through $[-]_{<r} : \bSigma_\infty^\op\to\bSigma_r^\op$.
\end{lemma}
\begin{proof}
    We prove the case where $r = 0$; the rest follow by induction with \cref{prop:infty-sheaf}.
    When $r = 0$, we are showing that $X : \bSigma_\infty^\op\to\sX$ is a sheaf of $(\infty, 0)$-categories precisely if it is essentially constant.
    
    Note that $W_{(\infty, 0)}$ is generated by $\Sigma^m\Sp_U[\vec k]\hookrightarrow\Sigma^m\Delta_U[\vec k]$ and $\Sigma^m\Delta_U[1]\to\Sigma^m\Delta_U[0]$ for $U\in\sC$, $m\geq0$, and $[\vec k]\in\bSigma_\infty$.
    By induction with \cref{prop:spine-redux}, $W_{(\infty, 0)}$ therefore contains the morphisms $\Delta_U[\vec k, 1]\to\Delta_U[\vec k]$ for all $[\vec k]\in\bSigma_\infty$, as well as the morphisms
    $$
    \underbrace{\Delta_U[\vec k, 1]\sqcup_{\Delta_U[\vec k]}\dots\sqcup_{\Delta_U[\vec k]}\Delta_U[\vec k, 1]}_\ell \simeq \Delta[\vec k,\Sp_U[\ell]] \to \Delta_U[\vec k, \ell]
    $$
    for all $[\vec k]\in\bSigma_\infty$ and $\ell\geq0$.
    This implies that $W_{(\infty, 0)}$ contains all maps of the form $\Delta_U[\vec k,\ell]\to\Delta_U[\vec k]$, implying that $X$ is $W_{(\infty, 0)}$-local if and only if $X$ is essentially constant, as desired.
\end{proof}

\begin{proposition}\label{prop:kappa}
    For $-2\leq n\leq\infty$ and $0\leq r < n+2$, the fully faithful inclusion $\Sh_{(n, r)}(\sC)\subseteq\Sh_{(n, r+1)}(\sC)$ admits a right adjoint $\kappa_{\leq r}$.
\end{proposition}
\begin{proof}
    First suppose $n = \infty$.
    The right adjoint to the inclusion $\Fun(\bSigma_r^\op, \sX) \to \Fun(\bSigma_{r+1}^\op, \sX)$ is given by restriction to $\bSigma_r\subseteq\bSigma_{r+1}$ (see \cref{rem:r-poly-sheaf}).
    By \cref{lem:truncate}, restriction descends to a right adjoint $\kappa_{\leq r}:\Sh_{(\infty, r+1)}(\sC) \to \Sh_{(\infty, r)}(\sC)$.
    The general case follows from observing that the above right adjoint descends further to a right adjoint $\Sh_{(n, r+1)}(\sC) \to \Sh_{(n, r)}(\sC)$ for $n < \infty$.
\end{proof}
\begin{remark}\label{rem:pi}
    On the other hand, it follows by definition that every inclusion $\Sh_{(n, r)}(\sC) \subseteq\Sh_{(n', r')}(\sC)$ with $n\leq n'$ and $r\leq r'$ admits a left adjoint $\pi_{\leq r}$.
\end{remark}

\subsection{Sheaves of \texorpdfstring{$(\infty, \infty)$}{(oo, oo)}-categories}
\label{ssec:sh-infty}

\begin{proposition}\label{prop:kappa-limit}
    The category $\Sh_{(\infty, \infty)}(\sC)$ is equivalent to the limit
    $$
    \varlim\left(\dots\to\Sh_{(\infty, 2)}(\sC) \xrightarrow{\kappa_{\leq1}} \Sh_{(\infty, 1)}(\sC) \xrightarrow{\kappa_{\leq0}} \Sh_{(\infty, 0)}(\sC)\right)
    $$
    in $\hat{\Cat_\infty}$.
\end{proposition}
\begin{proof}
    Since $\bSigma_\infty \simeq \varcolim_r\bSigma_r$, we have an equivalence $\Fun(\bSigma_\infty^\op,\sX) \simeq \varlim_r\Fun(\bSigma_r^\op, \sX)$.
    This allows us to identify the limit $\varlim_r\Sh_{(\infty, r)}(\sC)$ of right adjoints with the full subcategory of $\Fun(\bSigma_\infty^\op, \sX)$ on those $X:\bSigma_\infty^\op\to\sX$ that restrict on $\bSigma_r$ to a sheaf of $(\infty, r)$-categories for every $0\leq r<\infty$, but these are precisely the sheaves of $(\infty, \infty)$-categories.
\end{proof}

One may also be interested in the sheaves obtained as as the analogous limit of left adjoints:

\begin{definition}\label{def:shomega}
    As in \cref{rem:pi}, the inclusions $\Sh_{(\infty, r)}(\sC)\hookrightarrow\Sh_{(\infty, r+1)}(\sC)$ admit left adjoints $\pi_{\leq r}$.
    Define the category of \emph{sheaves of $\omega$-categories} over $\sC$ to be the analogous limit
    $$
    \Sh_\omega(\sC) := \varlim\left(\dots\to\Sh_{(\infty, 2)}(\sC) \xrightarrow{\pi_{\leq1}} \Sh_{(\infty, 1)}(\sC) \xrightarrow{\pi_{\leq0}} \Sh_{(\infty, 0)}(\sC)\right)
    $$
    in $\hat{\Cat_\infty}$.
\end{definition}
\begin{remark}\label{rem:shomega}
    As $\Sh_\omega(\sC)$ is given as a limit of left adjoints between locally presentable categories, it is locally presentable.
    Moreover, we have by \cite[Proposition 5.5.3.13 and Corollary 5.5.3.4]{lurie} that it can be computed equivalently as the colimit
    $$
    \Sh_\omega(\sC) \simeq {\varcolim}^R\left(\Sh_{(\infty, 0)}(\sC) \subseteq \Sh_{(\infty, 1)}(\sC) \subseteq \Sh_{(\infty, 2)}(\sC)\subseteq\dots\right)
    $$
    in $\Pres_\infty^R$.
    
    Moreover, the right adjoint inclusions $\Sh_{(\infty, r)}(\sC) \subseteq \Sh_{(\infty, \infty)}(\sC)$ induce a canonical right adjoint functor $\Sh_\omega(\sC)\to\Sh_{(\infty, \infty)}(\sC)$, whose left adjoint we may denote by $\pi_{\leq\omega}$.
\end{remark}

To understand the relationship between $\Sh_\omega(\sC)$ and $\Sh_{(\infty, \infty)}$, we have the following more general result about colimits in $\Pres_\infty^R$.

\section{Distributors and complete Segal spaces}
\label{sec:distributors}

In this section, we relate the categories $\Sh_{(n, r)}(\sC)$ and $\Sh_\omega(\sC)$ constructed in \cref{sec:sheaves} to the general theory of complete Segal space objects in a distributor.

\subsection{Preliminaries}
\label{ssec:prelim}

We recall the necessary theory from \cite{lurie:infty2} below.

\begin{definition}\label{def:distributor}{\cite[Definition 1.2.1]{lurie:infty2}}
    Fix a category $\sY$ and a full subcategory $\sX$.
    Say that $\sY$ is an \emph{$\sX$-distributor} if the following hold:
    \begin{enumerate}[label={(D\arabic*)}]
        \item\label{D:loc-pres}
            $\sX$ and $\sY$ are locally presentable,
        \item\label{D:pi+kappa}
            The inclusion $\sX\subseteq\sY$ admits both a left adjoint $\pi_\sX$ called \emph{$\sX$-truncation}, and a right adjoint $\kappa_\sX$ called \emph{$\sX$-core},
        \item\label{D:univ-colim}
            For all $y\to x$ in $\sY$ with $x\in\sX$, the pullback functor $\sX_{/x}\to\sY_{/y}$ preserves colimits, and
        \item\label{D:locality}
            The functor $\chi : \sX \to (\hat{\Cat_\infty})^\op, x \mapsto \sY_{/x}$, which sends morphisms in $\sX$ to pullback functors, preserves limits.
    \end{enumerate}
\end{definition}
\begin{remark}\label{rem:topos}{\cite[Example 1.2.3 and Remark 1.2.6]{lurie:infty2}}
    If $\sY$ is an $\sX$-distributor, then $\sX$ is necessarily a topos.
    Conversely, every topos $\sX$ is a distributor relative to itself.
\end{remark}

\begin{definition}\label{def:Grpd+Cat}
    For $\sX$ a finitely complete category, let $\Cat(\sX)$ denote the full subcategory of $\Fun(\bDelta^\op, \sX)$ spanned by those $X : \bDelta^\op \to \sX$ satisfying the \emph{Segal condition}; that is, the canonical maps
    $$
    X_n \to X_1\times_{X_0}\dots\times_{X_0}X_1
    $$
    are equivalences for every $n\geq0$.
    
    Further, define $\Grpd(\sX)$ to be the full subcategory of $\Cat(\sX)$ spanned by those $X_\bullet$ such that whenever $[n] = S\cup S'$ with $S\cap S' = \{i\}$, the induced map
    $$
    X_n \to X_S\times_{X_{\{i\}}}X_{S'}
    $$
    is an equivalence.
\end{definition}
\begin{remark}\label{rem:Grpd+Cat}
    If $\sX$ is locally presentable, then so are $\Cat(\sX)$ and $\Grpd(\sX)$.
    Indeed, $\Cat(\sS)$ is the strongly reflective localisation of $\sP(\bDelta)$ at the maps $\Delta[1]\sqcup_{\Delta[0]}\dots\sqcup_{\Delta[0]}\Delta[1]\hookrightarrow\Delta[n]$ for $n\geq0$, and $\Grpd(\sS)$ is similarly a strongly reflective localisation of $\Cat(\sS)$.
    As every locally presentable category is a strongly reflective localisation of a presheaf category, it follows likewise that $\Cat(\sX)$ is a strongly reflective localisation of $\Fun(\bDelta^\op,\sX)$, and $\Grpd(\sX)$ is a strongly reflective localisation of $\Cat(\sX)$.
\end{remark}

\begin{definition}\label{def:Cat}\cite[Definition 1.2.7]{lurie:infty2}
    Fix a topos $\sX$ and an $\sX$-distributor $\sY$.
    Then, define $\SS_\sX(\sY)$ to be the full subcategory of $\Cat(\sY)$ spanned by those $Y\in\Cat(\sY)$ such that $Y_0\in\sX$.
    Call such objects \emph{Segal space objects} in $\sX$.
\end{definition}
\begin{remark}\label{rem:Cat}
    As $\SS_\sX(\sY) \simeq \Cat(\sY)\times_\sX\sY$ is a fibre product of right adjoints between locally presentable categories, $\SS_\sX(\sY)$ is also locally presentable, and the inclusion $\SS_\sX(\sY)\subseteq\Cat(\sY)$ is a right adjoint.
\end{remark}

\begin{lemma}\label{lem:cores}{\cite[Proposition 1.1.14]{lurie:infty2}}
    If $\sX$ is finitely complete, then the inclusion $\Grpd(\sX)\subseteq\Cat(\sX)$ admits a right adjoint $X_\bullet\mapsto X_\bullet^\sim$.
\end{lemma}
\begin{proof}[Proof sketch.]
    If $\sX \simeq \sS$, then we construct the right adjoint explicitly.
    For $X_\bullet\in\Cat(\sS)$, define the $\Ho(\sS)$-enriched category $hX$ where its objects are the underlying points of $X_0$, and the mapping space $\Hom_{hX}(x, y)$ is the homotopy class of the fibre product $\{x\}\times_{X_0}X_1\times_{X_0}\{y\}$ in $\Ho(\sS)$.
    Then, we define $X_n^\sim$ to be the full subspace of $X_n$ spanned by those cells $\sigma\in X_n$ for which each $d_i\sigma$ descends to an isomorphism in $hX$.
    
    If $\sX \simeq \sP(\sC)$, then $\Cat(\sX) \simeq \Fun(\sC^\op, \Cat(\sS))$ and $\Grpd(\sX) \simeq \Fun(\sC^\op, \Grpd(\sX))$, so the right adjoint in this case is given pointwise by the right adjoint for $\sS$.
    
    For a general $\sX$, let $X_\bullet\in\Cat(\sX)$, and consider the object $jX_\bullet\in\Cat(\sP(\sX))$ obtained by pointwise post-composition with the Yoneda embedding $j:\sX\hookrightarrow\sP(\sX)$.
    Then, the reflection $(jX_\bullet)^\sim\in\Grpd(\sP(\sX))$ turns out to be pointwise representable, allowing it to descend to an object $X_\bullet^\sim\in\Grpd(\sX)$.
\end{proof}
\begin{corollary}\label{cor:Gp}{\cite[Notation 1.2.9]{lurie:infty2}}
    For an $\sX$-distributor $\sY$, the inclusion $\Grpd(\sX)\subseteq\SS_\sX(\sY)$ admits a right adjoint $\Gp : \SS_\sX(\sY) \to \Grpd(\sX)$.
\end{corollary}
\begin{remark}\label{rem:Gp}
    Explicitly, $\Gp_\bullet Y \simeq (\kappa_{\sX,*}Y_\bullet)^\sim$, where $\kappa_{\sX,*}$ applies the core pointwise.
    By the proof sketch of \cref{lem:cores}, the canonical map $\Gp_0Y\to Y_0$ is an equivalence for all $Y_\bullet\in\SS_\sX(\sY)$.
\end{remark}

\begin{definition}\label{def:CSS}
    Fix an $\sX$-distributor $\sY$.
    Say that a morphism $f:Y_\bullet\to Y_\bullet'$ in $\SS_\sX(\sY)$ is a \emph{Segal equivalence} if the following conditions hold:
    \begin{enumerate}[label={(E\arabic*)}]
        \item\label{E:ff}
            $f$ is \emph{fully faithful}, in that
            $$
            \begin{tikzcd}
                Y_1 \ar[r]\ar[d]
                & Y_1' \ar[d] \\
                Y_0\times Y_0 \ar[r]
                & Y_0'\times Y_0'
            \end{tikzcd}
            $$
            is a pullback square, and
        \item\label{E:es}
            $f$ is \emph{essentially surjective} in that the induced map $|\Gp_\bullet Y| \to |\Gp_\bullet Y'|$ is an equivalence in $\sX$.
    \end{enumerate}
    Say that $Y_\bullet \in \SS_\sX(\sY)$ is a \emph{complete Segal space} if it is local with respect to the Segal equivalences.
    Denote by $\CSS_\sX(\sY)$ the full subcategory of $\SS_\sX(\sY)$ spanned by the complete Segal spaces.
\end{definition}
\begin{remark}\label{rem:CSS}
    By \cite[Theorem 1.2.13]{lurie:infty2}, $Y_\bullet\in\SS_\sX(\sY)$ is a complete Segal space if and only if $\Gp_\bullet Y$ is essentially constant; that is, $\Gp_\bullet Y$ lies in the essential image of the diagonal $\sX\subseteq\Grpd(\sX)$.
\end{remark}

\begin{proposition}\label{prop:CSS}{\cite[Proposition 1.3.2]{lurie:infty2}}
    If $\sY$ is an $\sX$-distributor, then the fully faithful diagonal functor $\sX\subseteq\CSS_\sX(\sY)$ exhibits $\CSS_\sX(\sY)$ as an $\sX$-distributor as well.
\end{proposition}
\begin{remark}\label{rem:CSSdist}
    The truncation functor $\pi_\sX:\CSS_\sX(\sY)\to\sX$, being left adjoint to the diagonal, is given by geometric realisation $\pi_\sX Y \simeq |Y_\bullet|$.
    On the other hand, the core functor $\kappa_\sX:\CSS_\sX(\sY)\to\sX$ is given by $\kappa_\sX Y \simeq Y_0$.
\end{remark}

\subsection{The \texorpdfstring{$\infty$}{oo}-categories of distributors}
\label{ssec:dist}

\begin{definition}\label{def:dist-cats}
    Fix a topos $\sX$.
    Then, define $\Dist_\sX^L$ to be the subcategory of $(\hat{\Cat_\infty})_{\sX/}$ where the objects are $\sX$-distributors $\sX\subseteq\sY$, and a functor $\psi:\sZ\to\sY$ under $\sX$ is a morphism of $\Dist_\sX^L$ if and only if
    \begin{enumerate}[label={(L\arabic*)}]
        \item\label{L:cocts}
            $\psi$ preserves colimits, and
        \item\label{L:pi}
            $\psi$ preserves truncation, in that $\pi_\sX\psi\simeq\pi_\sX$.
    \end{enumerate}
    On the other hand, define $\Dist_\sX^R$ to be the subcategory of $(\hat{\Cat_\infty})_{\sX/}$ where the objects are again the $\sX$-distributors $\sX\subseteq\sY$, and a functor $\phi:\sY\to\sZ$ under $\sX$ is a morphism of $\Dist_\sX^R$ if and only if
    \begin{enumerate}[label={(R\arabic*)}]
        \item\label{R:cts}
            $\phi$ preserves limits and $\lambda$-filtered colimits for some $\lambda\gg0$, and
        \item\label{R:kappa}
            $\phi$ preserves cores, in that $\kappa_\sX\phi\simeq\kappa_\sX$.
    \end{enumerate}
\end{definition}
\begin{remark}\label{rem:dist-cats}
    Since $\sX$-distributors are necessarily locally presentable, \ref{L:cocts} is equivalent by \cite[Corollary 5.5.2.9]{lurie} to asserting that $\psi$ admits a right adjoint $\phi$.
    As truncation is left adjoint to the inclusion of $\sX$, \ref{L:pi} is then equivalent to asserting that the right adjoint $\phi$ lies under $\sX$; that is, $\phi|_\sX\simeq\id_\sX$.
    
    Dually, \ref{R:cts} is likewise equivalent to asserting that $\phi$ admits a left adjoint $\psi$.
    As core is right adjoint to the inclusion of $\sX$, \ref{R:kappa} is equivalent to asserting that the left adjoint $\psi$ lies under $\sX$; that is $\psi|_\sX\simeq\id_\sX$.
    This demonstrates that the two categories of \cref{def:dist-cats} are formally dual: $\Dist_\sX^L\simeq(\Dist_\sX^R)^\op$.
\end{remark}

\begin{proposition}\label{prop:distRlimits}
    For any topos $\sX$, the functor $\Dist_\sX^R\to(\hat{\Cat_\infty})_{/\sX}$ sending an $\sX$-distributor $\sY$ to its core $\kappa_\sX:\sY\to\sX$ creates all small limits.
\end{proposition}
\begin{proof}
    Let $\sY_\bullet : K \to \Dist_\sX^R$ be a small diagram indexed by a simplicial set $K$.
    The limit of the diagram $K\to(\hat{\Cat_\infty})_{/\sX}$ can be computed as the limit of the corresponding diagram $K^\triangleright\to\hat{\Cat_\infty}$ that sends the cocone point of $K^\triangleright$ to $\sX$ via core maps $\kappa_\sX:\sY_k\to\sX$.
    By \cite[Theorem 5.5.3.18]{lurie}, we can compute this limit instead as a limit of the diagram $\tilde\sY_\bullet : K^\triangleright\to\Pres_\infty^R$.
    
    Let $\sY := \varlim_k\tilde\sY_k$, then $\sY$ is locally presentable, and the projection map $\kappa_\sX:\sY\to\sX$ is a right adjoint.
    The left adjoint coincides with the functor $\sX\to\sY\simeq\varlim_k\tilde\sY_k$ in $\hat{\Cat_\infty}$ induced by the inclusions $\sX\subseteq\sY_k$ for $k\in K^\triangleright$.
    In particular, since these inclusions are fully faithful right adjoints, the same is true for the left adjoint $\sX\to\sY$.
    This proves that $\sY$ satisfies \ref{D:loc-pres} and \ref{D:pi+kappa}.
    
    Let $y\to x$ in $\sY$ with $x\in\sX$.
    Note that $\sY_{/y}\simeq\varlim_k(\tilde\sY_k)_{/y_k}$, where $y_k$ is the image of $y$ under the projection $\sY\to\sY_k$, and the map $\sX_{/x}\to\sY_{/y}$ is the canonical map induced by the functors $\sX_{/x}\to(\tilde\sY_k)_{/y_k}$ for each $k\in K^\triangleright$.
    Since each $\sY_k$ is an $\sX$-distributor, the functor $\sX_{/x}\to(\tilde\sY_k)_{/y_k}$ is a left adjoint for every $k\in K^\triangleright$.
    By \cite[Proposition 5.5.3.13]{lurie}, $\sX_{/x}\to\sY_{/y}$ must be a left adjoint as well, establishing \ref{D:univ-colim}.
    
    For $k\in K^\triangleright$, let $\chi^k:\sX\to(\hat{\Cat_\infty})^\op$ denote the functor mapping $x\in\sX$ to the category $(\tilde\sY_k)_{/x}$.
    Since each $\chi^k$ preserves limits from $\tilde\sY_k$ being an $\sX$-distributor, the same is true for the induced functor $\vec\chi : \sX \to \Fun(K^\triangleright, (\hat{\Cat_\infty})^\op)$ that sends each $x\in\sX$ to the functor mapping $k\mapsto(\tilde\sY_k)_{/x}$.
    Now, the functor $\chi : \sX \to (\hat{\Cat_\infty})^\op$ sending $x\mapsto\sY_{/x}\simeq\varlim_k(\tilde\sY_k)_{/x}$ is precisely the composite of limit-preserving functors
    $$
    \sX \xrightarrow{\vec\chi} \Fun(K^\triangleright, (\hat{\Cat_\infty})^\op) \xrightarrow{\varlim} (\hat{\Cat_\infty})^\op
    $$
    and thus preserves limits, establishing \ref{D:locality}.
    
    Therefore, the limit $\sY$ is an $\sX$-distributor, and the canonical projections $\sY\to\sY_k$ are core-preserving right adjoints.
    Suppose we have a cone from an $\sX$-distributor $\sZ$ to $\sY_\bullet$ in $\Dist_\sX^R$.
    This induces a cone from $\sZ$ to $\tilde\sY_\bullet$ in $\Pres_\infty^R$, inducing an essentially unique right adjoint $\sZ\to\sY$.
    Moreover, since the functors $\sZ\to\sY_k$ lie under $\sX$, the same is true for $\sZ\to\sY$, ensuring that the canonical functor $\sZ\to\sY$ is a morphism of $\Dist_\sX^R$, proving that $\sY$ is indeed a limit of $\sY_\bullet$ in $\Dist_\sX^R$.
\end{proof}

\begin{proposition}\label{prop:distLlimits}
    For any topos $\sX$, the functor $\Dist_\sX^L\to(\hat{\Cat_\infty})_{/\sX}$ sending an $\sX$-distributor $\sY$ to its truncation $\pi_\sX:\sY\to\sX$ creates all small limits.
\end{proposition}
\begin{proof}
    Let $\sY_\bullet : K \to \Dist_\sX^L$ be a small diagram indexed by a simplicial set $K$.
    The limit of the diagram $K \to (\hat{\Cat_\infty})_{/\sX}$ can be computed as a limit of the corresponding diagram $K^\triangleright\to\hat{\Cat_\infty}$ that sends the cocone point of $K^\triangleright$ to $\sX$ via truncation maps $\pi_\sX:\sY_k\to\sX$.
    By \cite[Proposition 5.5.3.13]{lurie}, we can compute this limit instead as a limit of the diagram $\tilde\sY_\bullet:K^\triangleright\to\Pres_\infty^L$.
    
    Let $\sY := \varlim_k\tilde\sY_k$, then $\sY$ is locally presentable, and the projection map $\pi_\sX:\sY\to\sX$ is a left adjoint.
    The right adjoint coincides with the functor $\sX\to\sY\simeq\varlim_k\tilde\sY_k$ in $\hat{\Cat_\infty}$ induced by the inclusions $\sX\subseteq\sY_k$ for $k\in K^\triangleright$.
    In particular, since these inclusions are fully faithful left adjoints, the same is true for the right adjoint $\sX\to\sY$.
    This proves that $\sY$ satisfies \ref{D:loc-pres} and \ref{D:pi+kappa}.
    Moreover, that $\sY$ satisfies \ref{D:univ-colim} and \ref{D:locality} follows the argument in \cref{prop:distRlimits} verbatim.
    
    Therefore, the limit $\sY$ is an $\sX$-distributor, and the canonical projections $\sY\to\sY_k$ are truncation-preserving left adjoints.
    Suppose we have a cone from an $\sX$-distributor $\sZ$ to $\sY_\bullet$ in $\Dist_\sX^L$.
    This induces a cone from $\sZ$ to $\tilde\sY_\bullet$ in $\Pres_\infty^L$, inducing an essentially unique left adjoint $\sZ\to\sY$.
    Moreover, since the functors $\sZ\to\sY_k$ lie under $\sX$, the same is true for $\sZ\to\sY$, ensuring that the canonical functor $\sZ\to\sY$ is a morphism of $\Dist_\sX^L$, proving that $\sY$ is indeed a limit of $\sY_\bullet$ in $\Dist_\sX^L$.
\end{proof}

\begin{remark}\label{rem:completeness}
    From the duality in \cref{rem:dist-cats}, $\Dist_\sX^R$ has all small limits and colimits (and likewise for $\Dist_\sX^L$).
    By \cite[Lemma 2.3.11]{goldthorpe:fixed}, the forgetful functor $\Dist_\sX^R\to\Pres_\infty^R$ creates all small limits indexed by a weakly contractible simplicial set; that is, weakly contractible limits of $\sX$-distributors can be computed on their underlying categories.
    On the other hand, the topos $\sX$ viewed as a distributor over itself is simultaneously the initial and terminal object in $\Dist_\sX^R$.
\end{remark}

\subsection{Functoriality of complete Segal spaces}
\label{ssec:css}

\begin{definition}\label{def:Xprod}
    Let $\sY$ be a category with full subcategory $\sX\subseteq\sY$.
    Say that a functor $\phi:\sY\to\sZ$ \emph{preserves fibre products over $\sX$} if for every fibre product $y^1\times_xy^2$ in $\sY$ with $x\in\sX$, the induced map $\phi(y^1\times_xy^2)\to\phi(y^1)\times_{\phi(x)}\phi(y^2)$ is an equivalence in $\sZ$.
    
    If $\sX$ has finite limits, and $\sX\subseteq\sY$ preserves these limits, then say that a functor $\phi : \sY\to\sZ$ is \emph{relatively left exact} (relative to $\sX$) if $\phi| : \sX \to \sZ$ is left exact, and $\phi$ preserves fibre products over $\sX$.
\end{definition}

The functorial nature of the construction of complete Segal spaces can be summarised as follows:

\begin{lemma}\label{lem:CSSfun}
    Let $g_*:\sX\to\sX'$ be a geometric morphism of topoi.
    Fix an $\sX$-distributor $\sY$ and an $\sX'$-distributor $\sY'$, and let $\phi:\sY\to\sY'$ be a functor such that
    \begin{itemize}
        \item $\phi$ extends $g_*$, in that $\phi|_\sX\simeq g_*$,
        \item $\phi$ commutes with cores, in that $\phi\kappa_\sX\simeq\kappa_{\sX'}\phi$, and
        \item $\phi$ preserves fibre products over $\sX$.
    \end{itemize}
    Then, $\phi_*:\Fun(\bDelta^\op,\sY)\to\Fun(\bDelta^\op,\sY')$ restricts to a functor $\phi_*:\CSS_\sX(\sY)\to\CSS_{\sX'}(\sY')$.
\end{lemma}
\begin{remark}\label{rem:CSSfun}
    In the situation of \cref{lem:CSSfun}, $\phi_*:\CSS_\sX(\sY)\to\CSS_{\sX'}(\sY')$ acts pointwise on the underlying functors.
    Since limits of complete Segal spaces are computed pointwise, this ensures that $\phi_*$ preserves fibre products over $\sX$.
    Likewise, \cref{rem:CSSdist} ensures that $\phi_*$ commutes with cores as well.
    Therefore, \cref{lem:CSSfun} combined with \cref{prop:CSS} implies an \emph{endofunctorial} nature of the construction of complete Segal spaces.
\end{remark}
\begin{proof}[Proof of \cref{lem:CSSfun}.]
    Since $\phi$ is relatively left exact, it follows that $\phi_*$ restricts to a functor $\SS_\sX(\sY)\to\Cat_{\sX'}(\sY')$.
    Similarly, the left exactness of the left adjoint $g^*\dashv g_*$ induces a commutative square of left adjoints
    $$
    \begin{tikzcd}
        \Grpd(\sX') \ar[r, hook]\ar[d, "g^*"']
        & \Cat(\sX') \ar[d, "g^*"] \\
        \Grpd(\sX) \ar[r, hook]
        & \Cat(\sX)
    \end{tikzcd}
    $$
    Taking right adjoints implies for any $X_\bullet\in\Cat(\sX)$ that $(g_*X_\bullet)^\sim \simeq g_*(X_\bullet^\sim)$.
    It therefore follows for any $Y_\bullet\in\SS_\sX(\sY)$ that
    \[
        \Gp_\bullet(\phi_*Y) &\simeq ((\kappa_{\sX'}\circ\phi)_*Y)^\sim \tag{\cref{rem:Gp}} \\
            &\simeq ((\phi\circ\kappa_\sX)_*Y)^\sim \tag{$\phi$ commutes with cores} \\
            &\simeq ((g_*\circ\kappa_\sX)_*Y)^\sim \tag{$\phi$ extends $g_*$} \\
            &\simeq g_*(\kappa_{\sX,*}Y)^\sim \simeq g_*(\Gp_\bullet Y) \tag{\cref{rem:Gp}}
    \]
    In particular, if $Y_\bullet$ is a complete Segal space, then $\Gp_\bullet(\phi_*Y) \simeq g_*\Gp_\bullet Y$ is essentially constant, ensuring that $\phi_*Y_\bullet$ is a complete Segal space by \cref{rem:CSS}.
\end{proof}

We make the functoriality of the construction of complete Segal spaces precise by realising the construction as a functor on the categories of distributors defined in \cref{ssec:dist}.
For this, we need a preliminary lemma:

\begin{lemma}\label{lem:colim-stability}
    Fix a topos $\sX$, and let $\sY$ be an $\sX$-distributor.
    Choose $\lambda\gg0$ such that $\sY$ is locally $\lambda$-presentable and the core functor $\kappa_\sX : \sY\to\sX$ is $\lambda$-accessible.
    Then,
    \begin{enumerate}[label={(\roman*)}]
        \item\label{it:colim-stability}
            $\CSS_\sX(\sY)$ is stable under $\lambda$-filtered colimits in $\Fun(\bDelta^\op, \sY)$.
        \item\label{it:colim-exactness}
            $\lambda$-filtered colimits in $\CSS_\sX(\sY)$ commute with $\lambda$-small limits.
        \item\label{it:core-accessible}
            The core functor $\kappa_\sX : \CSS_\sX(\sY) \to \sX$ preserves $\lambda$-filtered colimits.
    \end{enumerate}
\end{lemma}
\begin{proof}
    This follows the proof of \cite[Proposition 1.2.29]{lurie:infty2} \emph{mutatis mutandis}, noting that $\lambda$-filtered colimits commute with $\lambda$-small limits in any locally $\lambda$-presentable category.
\end{proof}

\begin{proposition}\label{prop:CSSfun}
    For a topos $\sX$, the construction of complete Segal spaces on $\sX$-distributors defines a functor $\CSS_\sX : \Dist_\sX^R \to \Dist_\sX^R$, where the functor $\phi_*:\CSS_\sX(\sY)\to\CSS_\sX(\sZ)$ induced by any $\phi:\sY\to\sZ$ in $\Dist_\sX^R$ acts pointwise.
\end{proposition}
\begin{proof}
    Let $\phi:\sY\to\sZ$ be a morphism of $\Dist_\sX^R$.
    Then, \cref{lem:CSSfun} implies that pointwise application of $\phi$ defines a functor $\phi_*:\CSS_\sX(\sY)\to\CSS_\sX(\sZ)$.
    Moreover, this functor preserves cores and small limits.
    
    Choose $\lambda\gg0$ so that $\phi:\sY\to\sZ$ is a $\lambda$-accessible functor between locally $\lambda$-presentable categories.
    Then, \cref{lem:colim-stability} implies that $\lambda$-filtered colimits in $\CSS_\sX(\sY)$ and $\CSS_\sX(\sZ)$ are computed pointwise, and are thus preserved by $\phi_*$.
    Therefore, $\phi_*$ is indeed a morphism of $\Dist_\sX^R$.
\end{proof}
\begin{remark}\label{rem:CSSLfun}
    By \cref{rem:dist-cats}, the construction of complete Segal spaces on $\sX$-distributors also defines an endofunctor on $\Dist_\sX^L$.
    However, the functor $\psi_! : \CSS_\sX(\sZ)\to\CSS_\sX(\sY)$ induced by $\psi:\sZ\to\sY$ does not act pointwise; rather, if $\Seg_\sX : \Fun(\bDelta^\op, \sY) \to \CSS_\sX(\sY)$ denotes a left adjoint to the inclusion, then $\psi_!$ is given on $Z_\bullet\in\CSS_\sX(\sZ)$ by $(\psi_!Z)_\bullet \simeq \Seg_\sX(\psi_*Z_\bullet)$.
\end{remark}

\begin{theorem}\label{thm:CSS-R-cts}
    For a topos $\sX$, the functor $\CSS_\sX : \Dist_\sX^R \to \Dist_\sX^R$ preserves weakly contractible limits.
\end{theorem}
\begin{proof}
    Let $\sY_\bullet : K \to \Dist_\sX^R$ be a diagram indexed by a weakly contractible simplicial set $K$, and let $\sY := \varlim_k\sY_k$, with canonical projection maps $\varpi_k:\sY\to\sY_k$.
    Consider the commutative square
    $$
    \begin{tikzcd}
        \CSS_\sX(\sY) \ar[r]\ar[d, hook]
        & \varlim_k\CSS_\sX(\sY_k) \ar[d, hook] \\
        \Fun(\bDelta^\op, \sY) \ar[r, "\sim"']
        & \varlim_k\Fun(\bDelta^\op, \sY_k)
    \end{tikzcd}
    $$
    where the limits are computed in $\hat{\Cat_\infty}$.
    By \cref{rem:completeness}, the limit $\sY = \varlim_k\sY_k$ can be computed on underlying categories in $\hat{\Cat_\infty}$, implying that the bottom arrow in the above diagram is an equivalence.
    \Cref{rem:completeness} also implies that it suffices to show that the top arrow in the above diagram is an equivalence as well.
    As all other functors are fully faithful, the top arrow is fully faithful.
    Therefore, it remains to show that the top arrow is essentially surjective.
    
    From the commutative diagram, we can identify $\varlim_k\CSS_\sX(\sY_k)$ with  the full subcategory of $\Fun(\bDelta^\op, \sY)$ spanned by those $Y:\bDelta^\op\to\sY$ for which $\varpi_kY : \bDelta^\op \to \sY_k$ lies in $\CSS_\sX(\sY_k)$ for every $k\in K$.
    It suffices to show that any such functor $Y:\bDelta^\op\to\sY$ is already a complete Segal space in $\sY$.
    
    Certainly $Y_0\in\sX$, since $\varpi_kY_0\in\sX$ for all $k\in K$.
    Since $\varpi_k$ is a morphism of $\Dist_\sX^R$, it preserves limits.
    In particular, for $n\geq0$, the canonical map
    $$
    Y_n \to \underbrace{Y_1\times_{Y_0}\dots\times_{Y_0}Y_1}_n
    $$
    in $\sY$ descends via $\varpi_k$ to an equivalence in $\sY_k$ for every $k\in K$, and is therefore also an equivalence.
    In particular, $Y_\bullet\in\SS_\sX(\sY)$.
    
    As $K$ is weakly contractible, it is nonempty, so fix some $k_0\in K$.
    Since $\varpi_{k_0}$ preserves cores and limits, the proof of \cref{lem:CSSfun} implies that $\Gp_\bullet(\varpi_{k_0,*}Y_\bullet) \simeq \Gp_\bullet Y$.
    Since $\Gp_\bullet(\varpi_{k_0,*}Y_\bullet)$ is essentially constant, this proves that $Y_\bullet$ is a complete Segal space, as desired.
\end{proof}

An analogue to \cref{thm:CSS-R-cts} for colimits is likely untrue in general.

\subsection{Distributors of sheaves of higher categories}
\label{ssec:univ}

As in \cref{sec:sheaves}, fix a topos $\sX$, presented as a left exact accessible localisation of $\sP(\sC)$ for some small category $\sC$, where the localisation functor is denoted $(-)^\#:\sP(\sC)\to\sX$.
In this section, we aim to use the language of distributors and complete Segal spaces to further study the categories of sheaves of $(n, r)$-categories over $\sC$, particularly when $n=\infty$.

\begin{lemma}\label{lem:compare}
    Let $\sY$ be an $\sX$-distributor.
    Then, $Y_\bullet\in\SS_\sX(\sY)$ is a complete Segal space if and only if it is local with respect to $|\bI|_U^\#\to\Delta_U[0]$ for every $U\in\sC$.
\end{lemma}
\begin{proof}
    Since $\Gp_\bullet Y\in\Grpd(\sX)$ satisfies the Segal condition, it is essentially constant if and only if the map $\Gp_1Y\to\Gp_0Y$ is an equivalence, which is equivalent to asserting that the underlying functor $\Gp_\bullet Y : \bDelta^\op\to\sX$ is local with respect to $\Delta_U[1]\to\Delta_U[0]$ for every $U\in\sC$.
    Now, note that
    \[
        \Hom_{\Fun(\bDelta^\op,\sX)}(\Delta_U[1], \Gp_\bullet Y) &\simeq \Hom_{\sP(\bDelta)}(\Delta[1], \Gp_\bullet Y(U)) \tag{\cref{def:poly-sheaf}} \\
            &\simeq \Hom_{\Cat(\sS)}(\Delta[1], \Gp_\bullet Y(U)) \tag{both $\Delta[1]$ and $\Gp_\bullet Y(U)$ satisfy the Segal condition} \\
            &\simeq \Hom_{\Grpd(\sS)}(\bI, \Gp_\bullet Y(U)) \tag{$\bI$ is the groupoidification of $\Delta[1]$} \\
            &\simeq \Hom_{\sP(\bDelta)}(\bI, \Gp_\bullet Y(U)) \tag{$\Grpd(\sS)\subseteq\sP(\bDelta)$ is fully faithful} \\
            &\simeq \Hom_{\Fun(\bDelta^\op,\sX)}(|\bI|_U^\#, \Gp_\bullet Y) \tag{\cref{def:loc-realisation}} \\
            &\simeq \Hom_{\Grpd(\sX)}(|\bI|_U^\#, \Gp_\bullet Y) \tag{both $|\bI|_U^\#$ and $\Gp_\bullet Y$ are groupoid objects in $\sX$} \\
            &\simeq \Hom_{\SS_\sX(\sY)}(|\bI|_U^\#, Y_\bullet) \tag{\cref{cor:Gp}}
    \]
    and similarly for $\Delta_U[0]$.
    Therefore, $\Gp_\bullet Y$ is essentially constant if and only if $Y_\bullet$ is local with respect to $|\bI|_U^\#\to\Delta_U[0]$ for every $U\in\sC$, as desired.
\end{proof}
\begin{corollary}\label{cor:compare}
    The category $\Sh_{(\infty, r)}(\sC)$ is an $\sX$-distributor for every $0\leq r<\infty$.
    Moreover, we have a canonical equivalence $\Sh_{(\infty, r+1)}(\sC) \simeq \CSS_\sX(\Sh_{(\infty, r)}(\sC))$.
\end{corollary}
\begin{proof}
    This is clear if $r = 0$, since $\Sh_{(\infty, 0)}(\sC) \simeq \sX$.
    The result for $r > 0$ follows by induction from combining \cref{lem:compare} with \cref{prop:infty-sheaf}.
\end{proof}

That iterated application of $\CSS_\sX$ to $\sX$ yields sheaves of $(\infty, r)$-categories was already suggested in \cite[Variant 1.3.8]{lurie:infty2}.
We are more interested in the case $r = \infty$.
From the inductive formula in \cref{cor:compare}, we might expect that $\Sh_{(\infty, \infty)}(\sC)$ is an $\sX$-distributor with a canonical equivalence $\Sh_{(\infty, \infty)}(\sC)\simeq\CSS_\sX(\Sh_{(\infty, \infty)}(\sC))$.
In fact, more is true:

\begin{theorem}\label{thm:CSS-terminal}
    Let $\mathbf{Fix}_{\CSS}^R$ denote the category whose objects are $\sX$-distributors $\sY$ with an equivalence $\sY\xrightarrow\sim\CSS_\sX(\sY)$, and whose morphisms are the morphisms of $\Dist_\sX^R$ that commute with these equivalences.
    Then, the terminal object of $\mathbf{Fix}_{\CSS}^R$ is equivalent to $\Sh_{(\infty, \infty)}(\sC)$.
\end{theorem}
\begin{proof}
    By \cref{prop:kappa-limit} and \cref{cor:compare}, we can compute $\Sh_{(\infty, \infty)}(\sC)$ as the limit
    $$
    \Sh_{(\infty, \infty)}(\sC) \simeq \varlim\left(\dots\to\CSS_\sX^3(\sX) \xrightarrow{\kappa_{\sX,*}} \CSS_\sX^2(\sX) \xrightarrow{\kappa_{\sX,*}} \CSS_\sX(\sX) \xrightarrow{\kappa_\sX} \sX\right)
    $$
    By \cref{rem:completeness}, this limit can be lifted to a limit in $\Dist_\sX^R$, proving that $\Sh_{(\infty, \infty)}(\sC)$ is indeed an $\sX$-distributor.
    Moreover, \cref{thm:CSS-R-cts} proves that the induced functor $\CSS_\sX(\Sh_{(\infty, \infty)}(\sC)) \to \Sh_{(\infty, \infty)}(\sC)$ is an equivalence.
    Since $\sX$ is the terminal object in $\Dist_\sX^R$, the theorem now follows from \cite[Corollaries 2.2.9 and 2.3.3]{goldthorpe:fixed}.
\end{proof}
\begin{remark}\label{rem:CSS-terminal}
    The crux of the above proof is the observation that $\Sh_{(\infty, \infty)}(\sC)$ can be realised as an instance of Ad\'amek's construction.
    In particular, this shows that $\Sh_{(\infty, \infty)}(\sC)$ has a stronger universal property: it is the terminal $\CSS_\sX$-coalgebra in $\Dist_\sX^R$.
\end{remark}

\begin{remark}\label{rem:absolute-dist}
    Let $\sC\simeq*$, so that $\sX\simeq\sS$.
    Then, $\sS$-distributors are also called \emph{absolute distributors}.
    Haugseng proved in \cite[Theorem 7.18]{haugseng} that there is an equivalence $\sY\Cat\simeq\CSS_\sS(\sY)$ between $\sY$-enriched categories and complete Segal spaces in $\sY$ for every absolute distributor $\sY$.
    
    Following \cite[Definition 3.0.1 and Remark 3.0.2]{goldthorpe:fixed}, this implies that we have an equivalence of categories $\Sh_{(n, r)}(*) \simeq \Cat_{(n, r)}$ for all $-2\leq n\leq\infty$ and finite $0\leq r\leq n$.
    Then, taking appropriate limits following \cite[Definition 3.0.4]{goldthorpe:fixed}, we moreover have equivalences $\Sh_{(\infty, \infty)}(*) \simeq \Cat_{(\infty, \infty)}$ and $\Sh_\omega(*)\simeq\Cat_\omega$.
    
    In particular, we see that both $\Cat_{(\infty, \infty)}$ and $\Cat_\omega$ are absolute distributors.
    Moreover, by \cite[Theorem 3.3.5]{goldthorpe:fixed}, $\Cat_{(\infty, \infty)}$ is simultaneously a universal fixed points with respect to enrichment and the construction of complete Segal spaces.
\end{remark}

\subsection{Geometric morphisms and sheafification}
\label{ssec:sheafification}

In this section, we prove comparison results between categories of sheaves.
We prove in \cref{cor:shv-functorial} that the construction $\Sh_{(n, r)}(-)$ or $\Sh_\omega(-)$ is functorial with respect to geometric morphisms, sending geometric morphisms to right adjoints with relatively left exact left adjoints.
We likewise prove that the truncation functors $\pi_{\leq r}:\Sh_{(\infty, r')}(\sC) \to \Sh_{(\infty, r)}(\sC)$ are relatively left exact in \cref{prop:truncation-Xprod,cor:truncation-Xprod}.
Finally, we show in \cref{thm:sheafification} that $\Sh_{(n, r)}(\sC)$ is a relatively left exact accessible localisation of $\Fun(\sC^\op, \Cat_{(n, r)})$, and likewise $\Sh_\omega(\sC)$ is a relatively left exact accessible localisation of $\Fun(\sC^\op, \Cat_\omega)$.

Throughout the section, fix a topos $\sX$, exhibited as a left exact accessible localisation of $\sP(\sC)$ for some small category $\sC$ throughout this section.

\begin{lemma}\label{lem:sifted-Xprod}
    Suppose $K$ is sifted in the sense of \cite[Definition 5.5.8.1]{lurie}.
    Then, for all topoi $\sX$, the functor $\varcolim:\Fun(K, \sX)\to\sX$ preserves fibre products over $\sX$, where $\sX$ is viewed as a full subcategory of $\Fun(K, \sX)$ spanned by the essentially constant functors.
\end{lemma}
\begin{proof}
    Suppose $F\to x\gets G$ is a cospan of functors $K\to\sX$, where $x\in\sX$ is viewed as a constant functor.
    Since $K$ is sifted, the diagonal $K\to K\times K$ is cofinal.
    Therefore,
    \[
        \left(\varcolim_{k\in K}F(k)\right)\times_x\left(\varcolim_{\ell\in K}G(\ell)\right) &\simeq \varcolim_{(k,\ell)\in K\times K}\Big(F(k)\times_x G(\ell)\Big) \tag{colimits are universal in $\sX$} \\
            &\simeq \varcolim_{j\in K}\Big(F(j)\times_xG(j)\Big) \tag{$K\to K\times K$ is cofinal}
    \]
    showing that $\varcolim : \Fun(K, \sX)\to\sX$ preserves the fibre product $F\times_xG$.
\end{proof}

\begin{lemma}\label{lem:Seg-Xprod}
    Fix an $\sX$-distributor $\sY$.
    Suppose a morphism $Y_\bullet\to Z_\bullet$ in $\SS_\sX(\sY)$ is a Segal equivalence.
    Then, for any cospan $Z_\bullet\to x\gets Z_\bullet'$ in $\SS_\sX(\sY)$ with $x\in\sX$, the induced map $Y\times_xZ'\to Z\times_xZ'$ is a Segal equivalence.
\end{lemma}
\begin{proof}
    The functor $Y\times_xZ'\to Z\times_xZ'$ is certainly fully faithful in the sense of \ref{E:ff}, since limits commute with limits.
    We need to show that the map is also essentially surjective in the sense of \ref{E:es}; that is, $|\Gp_\bullet(Y\times_xZ')|\to|\Gp_\bullet(Z\times_xZ')|$ is an equivalence in $\sX$.
    
    Since $\Gp_\bullet$ is a right adjoint that restricts to the identity on $\sX$, we have $\Gp_\bullet(Y\times_xZ')\simeq(\Gp_\bullet Y)\times_x(\Gp_\bullet Z')$ and likewise for $Z\times_xZ'$.
    On the other hand, since $\bDelta^\op$ is sifted by \cite[Lemma 5.5.8.4]{lurie}, geometric realisation also preserves fibre products over $\sX$ by \cref{lem:sifted-Xprod}.

    Therefore, $Y\times_xZ'\to Z\times_xZ'$ satisfies \ref{E:es} if and only if the map $|\Gp_\bullet Y|\times_x|\Gp_\bullet Z'|\to|\Gp_\bullet Z|\times_x|\Gp_\bullet Z'|$ is an equivalence, which follows from the fact that the is the base change of the map $|\Gp_\bullet Y|\to|\Gp_\bullet Z|$ along $|\Gp_\bullet Z'|\to x$, which is an equivalence since $Y\to Z$ is a Segal equivalence.
\end{proof}
\begin{proposition}\label{prop:Seg-Xprod}
    For an $\sX$-distributor $\sY$, the left adjoint $\Seg_\sX : \SS_\sX(\sY) \to \CSS_\sX(\sY)$ to inclusion preserves fibre products over $\sX$.
\end{proposition}
\begin{proof}
    Suppose $Y\to x\gets Y'$ is a cospan in $\SS_\sX(\sY)$, where $x\in\sX$.
    Since $(\Seg_\sX Y)\times_x(\Seg_\sX Y')$ is a limit of complete Segal spaces, it is also a complete Segal space.
    Therefore, it suffices to show that the map $Y\times_xY'\to(\Seg_\sX Y)\times_x(\Seg_\sX Y')$ is a Segal equivalence, as it would then induce an equivalence $\Seg_\sX(Y\times_xY')\to(\Seg_\sX Y)\times_x(\Seg_\sX Y')$.
    
    Since $Y\times_xY'\to(\Seg_\sX Y)\times_x(\Seg_\sX Y')$ factors as
    \[
        Y\times_xY' \to (\Seg_\sX Y)\times_xY' \to (\Seg_\sX Y)\times_x(\Seg_\sX Y')
    \]
    the fact that it is a Segal equivalence follows from \cref{lem:Seg-Xprod}, since $Y\to\Seg_\sX Y$ and $Y'\to\Seg_\sX Y'$ are Segal equivalences.
\end{proof}

\begin{lemma}\label{lem:geom2geom}
    In the situation of \cref{lem:CSSfun}, suppose $\phi:\sY\to\sY'$ admits a left adjoint $\psi$.
    Then, the induced functor $\phi_*:\CSS_\sX(\sY)\to\CSS_{\sX'}(\sY')$ admits a left adjoint $\psi_!$.
    If moreover $\psi:\sY'\to\sY$ preserves fibre products over $\sX'$, then the same is true for $\psi_!$.
\end{lemma}
\begin{proof}
    Choose $\lambda\gg0$ so that $\phi:\sY\to\sZ$ is a $\lambda$-accessible functor between locally $\lambda$-presentable categories.
    Then, \cref{lem:colim-stability} implies that $\lambda$-filtered colimits in $\CSS_\sX(\sY)$ and $\CSS_\sX(\sZ)$ are computed pointwise, and are thus preserved by $\phi_*$.
    In particular, \cite[Corollary 5.5.2.9]{lurie} implies that $\phi_*$ admits a left adjoint $\psi_!$.
    
    If $\psi$ preserves fibre products over $\sX'$, then the induced functor $\psi_* : \Cat_{\sX'}(\sY') \to \SS_\sX(\sY)$ is left adjoint to $\phi_* : \SS_\sX(\sY) \to \Cat_{\sX'}(\sY')$.
    Taking left adjoints of the functors in the commutative square
    $$
    \begin{tikzcd}
        \CSS_\sX(\sY) \ar[r, "\phi_*"]\ar[d, hook]
        & \CSS_{\sX'}(\sY') \ar[d, hook] \\
        \SS_\sX(\sY) \ar[r, "\phi_*"']
        & \Cat_{\sX'}(\sY')
    \end{tikzcd}
    $$
    yields the commutative square in the diagram
    $$
    \begin{tikzcd}
        \CSS_{\sX'}(\sY) \ar[r, hook] \ar[dr, "\id"']
        & \Cat_{\sX'}(\sY') \ar[r, "\psi_*"]\ar[d, "\Seg_{\sX'}"']
        & \SS_\sX(\sY) \ar[d, "\Seg_\sX"] \\
        & \CSS_{\sX'}(\sY') \ar[r, "\psi_!"']
        & \CSS_\sX(\sY)
    \end{tikzcd}
    $$
    where the vertical functors preserve fibre products over $\sX'$ by \cref{prop:Seg-Xprod}, and the inclusion preserves all limits from being a right adjoint.
    Therefore, $\psi_!$ preserves fibre products over $\sX'$ as well.
\end{proof}

\begin{corollary}\label{cor:shv-functorial}
    Given a geometric morphism $g_* : \Sh(\sC) \to \Sh(\sC')$ of topoi of sheaves, we have for all $-2\leq n\leq\infty$ and $0\leq r\leq n$ an adjunction
    \[
        g^* : \Sh_{(n, r)}(\sC') \rightleftarrows : \Sh_{(n, r)}(\sC) : g_*
    \]
    and similarly an adjunction
    \[
        g^* : \Sh_\omega(\sC') \rightleftarrows : \Sh_\omega(\sC) : g_*
    \]
    where each $g_*$ acts pointwise, and the corresponding left adjoints $g^*$ are relatively left exact.
\end{corollary}

\begin{lemma}\label{lem:filtered-cofinal}
    Let $K$ be a filtered simplicial set, and $(\sA^k\to\sB^k)_{k\in K}$ a diagram in $\Fun(\Delta[1], \Cat_\infty)$ consisting of cofinal functors.
    Then, the colimit of this diagram is again a cofinal functor.
\end{lemma}
\begin{proof}
    Let $\sA\to\sB$ denote the colimit of the diagram of cofinal functors.
    By Quillen's Theorem A \cite[Theorem 4.1.3.1]{lurie}, we show for any $B\in\sB$ that $\sA_{B/} := \sA\times_\sB\sB_{B/}$ is weakly contractible.
    
    Choose $k\in K$ so that $B\in\sB^k$; that is, so that $B$ is in the essential image of the coprojection $\sB^k\to\sB$.
    Since $K$ is filtered, the canonical functor $K_{k/}\to K$ is cofinal.
    Therefore,
    \[
        \sA\times_\sB\sB_{B/} &\simeq \left(\varcolim_\ell\sA^\ell\right)\times_{\left(\varcolim_\ell\sB^\ell\right)}\left(\varcolim_\ell\sB^\ell\right)_{B/} \\
            &\simeq \left(\varcolim_{k\to\ell}\sA^\ell\right)\times_{\left(\varcolim_{k\to\ell}\sB^\ell\right)}\left(\varcolim_{k\to\ell}\sB^\ell\right)_{B/} \tag{cofinality of $K_{k/}\to K$} \\
            &\simeq \left(\varcolim_{k\to\ell}\sA^\ell\right)\times_{\left(\varcolim_{k\to\ell}\sB^\ell\right)}\left(\varcolim_{k\to\ell}\sB^\ell_{B/}\right) \tag{$B\in\sB^k$} \\
            &\simeq \varcolim_{k\to\ell}\left(\sA^\ell\times_{\sB^\ell}\sB^\ell_{B/}\right) \tag{filtered colimits are left exact in $\Cat_\infty$}
    \]
    By assumption, every $\sA^\ell\times_{\sB^\ell}\sB^\ell_{B/}$ is weakly contractible for every $k\to\ell$, implying the same for $\sA\times_\sB\sB_{B/}$, as desired.
\end{proof}
\begin{corollary}\label{cor:filtered-sifted}
    The full subcategory of $\Cat_\infty$ spanned by sifted categories is closed under filtered colimits.
\end{corollary}

\begin{proposition}\label{prop:truncation-Xprod}
    For $0\leq r < r' \leq \infty$, the functor $\pi_{\leq r} : \Sh_{(\infty, r')}(\sC) \to \Sh_{(\infty, r)}(\sC)$ is relatively left exact.
\end{proposition}
\begin{proof}
    We prove the proposition by induction on $r$, so suppose first that $r = 0$.

    Let $\bDelta_\infty$ denote the full subcategory of $\bDelta^{\times\infty}$ spanned by those $[\vec k] = (k_0, k_1, k_2, \dots)$ such that $k_n=0$ for all $n\gg0$.
    For $0\leq r'\leq\infty$, denote by $\bDelta_{r'}$ the full subcategory of $\bDelta_\infty$ spanned by those $[\vec k]$ such that $k_n=0$ for all $n\geq r'$.
    Then, we have a fully faithful inclusion $\bSigma_{r'}\subseteq\bDelta_{r'}$ for every $0\leq r'\leq\infty$.
    
    The functor $\pi_{\leq0} : \Sh_{(\infty, r')}(\sC) \to \Sh_{(\infty, 0)}(\sC) \simeq \sX$ is left adjoint to the diagonal $\sX\subseteq\Sh_{(\infty, r')}(\sC)$.
    The composite of right adjoints
    $$
    \sX\subseteq\Sh_{(\infty, r')}(\sC) \subseteq \Fun(\bSigma_{r'}^\op, \sX) \subseteq \Fun(\bDelta_{r'}^\op, \sX)
    $$
    where the rightmost inclusion is right Kan extension along $\bSigma_{r'}\subseteq\bDelta_{r'}$ is precisely the diagonal inclusion, which implies that $\pi_{\leq0}$ is the restriction of the left adjoint $\varcolim : \Fun(\bDelta_{r'}^\op, \sX)\to\sX$ to the diagonal.
    By \cref{lem:sifted-Xprod}, it suffices to prove that $\bDelta_{r'}^\op$ is sifted for every $0\leq r'\leq\infty$.
    
    For $r'<\infty$, note that $\bDelta_{r'} \simeq \bDelta^{\times r'}$.
    Since $\bDelta^\op$ is sifted by \cite[Lemma 5.5.8.4]{lurie}, and \cite[Corollary 4.1.1.13 and Proposition 4.1.1.3(2)]{lurie} imply that sifted categories are closed under finite products, it follows that $\bDelta_{r'}^\op$ is sifted for $r'$ finite.
    Finally, since $\bDelta_\infty \simeq \varcolim(\bDelta_0\subseteq\bDelta_1\subseteq\bDelta_2\subseteq\dots)$, it follows from \cref{cor:filtered-sifted} that $\bDelta_\infty^\op$ is sifted as well.
    
    Now, let $r > 0$, and suppose all $\pi_{\leq(r-1)}$ preserve fibre products over $\sX$.
    Then,
    $$
    \pi_{\leq(r-1),*} : \Fun(\bDelta^\op,\Sh_{(\infty,r'-1)}(\sC))\to\Fun(\bDelta^\op,\Sh_{(\infty,r-1)}(\sC))
    $$
    restricts to a functor $\SS_\sX(\Sh_{(\infty,r'-1)}(\sC))\to\SS_\sX(\Sh_{(\infty,r-1)}(\sC))$ that also preserves fibre products over $\sX$.
    
    As $\pi_{\leq r}\simeq\pi_{\leq(r-1), !}$ is the image of $\pi_{\leq(r-1)}$ under $\CSS_\sX : \Dist_\sX^L\to\Dist_\sX^L$, and $\pi_{\leq(r-1)}$ preserves fibre products over $\sX$ by assumption, it follows by \cref{lem:geom2geom} that $\pi_{\leq r}$ preserves fibre products over $\sX$ as well.
\end{proof}
\begin{corollary}\label{cor:truncation-Xprod}
    The localisation functor $\pi_{\leq\omega} : \Sh_{(\infty, \infty)}(\sC) \to \Sh_\omega(\sC)$ is relatively left exact.
\end{corollary}

The remainder of this section is dedicated to proving the following result:

\begin{theorem}[Sheafification]\label{thm:sheafification}
    For $-2\leq n\leq\infty$ and $0\leq r\leq n+2$, the category $\Sh_{(n, r)}(\sC)$ includes fully faithfully into $\Fun(\sC^\op, \Cat_{(n, r)})$, and this inclusion admits a relatively left exact left adjoint.
    
    Likewise, $\Sh_\omega(\sC)$ includes fully faithfully into $\Fun(\sC^\op, \Cat_\omega)$, and the inclusion admits a relatively left exact left adjoint.
\end{theorem}
\begin{proof}
    Follows from combining \cref{prop:presheaves} below with \cref{lem:geom2geom}.
\end{proof}

\begin{lemma}\label{lem:presheaves}
    Let $\sY$ be an $\sX$-distributor.
    For every small category $\sK$, the category $\Fun(\sK^\op, \sY)$ is a distributor over $\Fun(\sK^\op, \sX)$.
    Moreover, $\CSS_{\Fun(\sK^\op, \sX)}(\Fun(\sK^\op, \sY)) \simeq \Fun(\sK^\op, \CSS_\sX(\sY))$.
\end{lemma}
\begin{proof}
    $\Fun(\sK^\op, \sX)$ and $\Fun(\sK^\op, \sY)$ are certainly locally presentable, verifying \ref{D:loc-pres}.
    The inclusion $\sX\subseteq\sY$ preserves small limits and colimits, and limits and colimits in functor categories are computed pointwise, which proves \ref{D:pi+kappa}.
    By \cite[Corollary 1.2.5]{lurie:infty2}, $\Fun(\sK^\op, \sY)$ is then a distributor over $\Fun(\sK^\op, \sX)$ if and only if the following holds:
    \begin{itemize}
        \item[($*$)] For any simplicial set $J$ and natural transformation $\bar\alpha : \bar p\Rightarrow \bar q : J^\triangleright \to \Fun(\sK^\op, \sY)$ such that $\bar q$ is a colimit diagram in $\Fun(\sK^\op, \sX)$ and the naturality squares of $\alpha = \bar\alpha|_K$ are pullback squares, then $\bar p$ is a colimit diagram in $\Fun(\sK^\op, \sY)$ if and only if the naturality squares of $\bar\alpha$ are pullback squares.
    \end{itemize}
    This condition can be checked pointwise, and \cite[Corollary 1.2.5]{lurie:infty2} implies that the analogous results hold for $\sX\subseteq\sY$.
    Therefore, $\Fun(\sK^\op, \sY)$ is then a distributor over $\Fun(\sK^\op, \sX)$, as desired.
    
    Since $\Fun(\bDelta^\op, \Fun(\sK^\op, \sY))\simeq\Fun(\sK^\op, \Fun(\bDelta^\op, \sY))$ and fibre products are computed pointwise in functor categories, we have $\Cat_{\Fun(\sK^\op, \sX)}(\Fun(\sK^\op, \sY)) \simeq \Fun(\sK^\op, \SS_\sX(\sY))$.
    Finally, since the conditions for a transformation to be a Segal equivalence can be checked pointwise, this equivalence descends to an equivalence $\CSS_{\Fun(\sK^\op, \sX)}(\Fun(\sK^\op, \sY)) \simeq \Fun(\sK^\op, \CSS_\sX(\sY))$.
\end{proof}

\begin{proposition}\label{prop:presheaves}
    Let $-2\leq n\leq\infty$ and $0\leq r\leq n+2$.
    If $\sX \simeq \sP(\sC)$, then we have an equivalence $\Sh_{(n, r)}(\sC) \simeq \Fun(\sC^\op, \Cat_{(n, r)})$.
    Similarly, $\Sh_\omega(\sC) \simeq \Fun(\sC^\op, \Cat_\omega)$.
\end{proposition}
\begin{proof}
    Suppose first that $n = \infty$.
    If $r = 0$, then there is nothing to prove, as $\Sh_{(\infty, 0)}(\sC) \simeq \sP(\sC)$ and $\Cat_{(\infty, 0)}\simeq\sS$.
    Then, iteratively applying \cref{lem:presheaves} proves the proposition if $r$ is finite.
    For $r = \infty$, we have
    \[
        \Sh_{(\infty, \infty)}(\sC) &\simeq \varlim\left(\dots\to\Sh_{(\infty, 2)}(\sC) \xrightarrow{\kappa_{\leq1}} \Sh_{(\infty, 1)}(\sC) \xrightarrow{\kappa_{\leq0}} \Sh_{(\infty, 0)}(\sC)\right) \\
            &\simeq \varlim\left(\dots\to\Fun(\sC^\op, \Cat_{(\infty, 2)}) \xrightarrow{\kappa_{\leq1,*}} \Fun(\sC^\op, \Cat_{(\infty, 1)}) \xrightarrow{\kappa_{\leq0,*}} \Fun(\sC^\op, \Cat_{(\infty, 0)})\right) \\
            &\simeq \Fun\left(\sC^\op, \varlim\left(\dots\to\Cat_{(\infty, 2)}\xrightarrow{\kappa_{\leq2}} \Cat_{(\infty, 1)} \xrightarrow{\kappa_{\leq1}} \Cat_{(\infty, 0)}\right)\right) \\
            &\simeq \Fun(\sC^\op, \Cat_{(\infty, \infty)})
    \]
    Similarly,
    \[
        \Sh_\omega(\sC) &\simeq \varlim\left(\dots\to\Sh_{(\infty, 2)}(\sC) \xrightarrow{\pi_{\leq1}} \Sh_{(\infty, 1)}(\sC) \xrightarrow{\pi_{\leq0}} \Sh_{(\infty, 0)}(\sC)\right) \\
            &\simeq \varlim\left(\dots\to\Fun(\sC^\op, \Cat_{(\infty, 2)}) \xrightarrow{\pi_{\leq1,*}} \Fun(\sC^\op, \Cat_{(\infty, 1)}) \xrightarrow{\pi_{\leq0,*}} \Fun(\sC^\op, \Cat_{(\infty, 0)})\right) \\
            &\simeq \Fun\left(\sC^\op, \varlim\left(\dots\to\Cat_{(\infty, 2)}\xrightarrow{\pi_{\leq2}} \Cat_{(\infty, 1)} \xrightarrow{\pi_{\leq1}} \Cat_{(\infty, 0)}\right)\right) \\
            &\simeq \Fun(\sC^\op, \Cat_\omega)
    \]
    Now suppose $n < \infty$.
    Note that $\Sh_{(n, r)}(\sC)$ is the localisation of $\Sh_{(\infty, r)}(\sC)\simeq\Fun(\sC^\op, \Cat_{(\infty, r)})$ at the maps $\Sigma^m(\bS^{k-m}\times\Delta_U[0])\to\Sigma^m\Delta_U[0]$ for $m\geq r$, $k > n$, and $U\in\sC$.
    A functor $F : \sC^\op\to\Cat_{(\infty, r)}$ is local with respect to $\Sigma^m(\bS^{k-m}\times\Delta_U[0])\to\Sigma^m\Delta_U[0]$ if and only if $F(U)\in\Cat_{(\infty, r)}$ is local with respect to $\Sigma^m\bS^{k-m}\to\Sigma^m*$.
    By \cite[Proposition 6.1.7]{gepner-haugseng}, it follows that $F$ lies in $\Sh_{(n, r)}(\sC)$ if and only if $F(U)\in\Cat_{(n, r)}$ for every $U\in\sC$, as desired.
\end{proof}

\end{document}